\documentclass{amsart}
\newtheorem{theorem}{Theorem}[section]
\newtheorem{lemma}[theorem]{Lemma}
\newtheorem{proposition}[theorem]{Proposition}
\newtheorem{corollary}[theorem]{Corollary}

\theoremstyle{definition}
\newtheorem{definition}[theorem]{Definition}
\newtheorem{example}[theorem]{Example}
\newtheorem{remark}[theorem]{Remark}

\numberwithin{equation}{section}

\usepackage{amsmath,amssymb,amsthm}
\def\Cmath{\mathbb{C}}
\def\Nmath{\mathbb{N}}
\def\Zmath{\mathbb{Z}}

\renewcommand{\blacksquare}{\sharp}

\begin{document}
\title{On  $\mathcal{I}$-null Lie algebras}
\author{L. Magnin}
\address{Institut de Math\'{e}matiques de Bourgogne,
UMR CNRS 5584, Universit\'{e} de Bourgogne, BP 47870, 21078 Dijon Cedex, France.}
\email{magnin@u-bourgogne.fr}

\subjclass[2010]{17B30}

\keywords{}
\date{}

\begin{abstract}
We  consider the class of complex Lie algebras for which the Koszul 3-form is zero,
and prove that it contains
all quotients of Borel subalgebras,
or of their nilradicals,
of finite dimensional complex semisimple Lie algebras.
A list of Kac-Moody types for indecomposable nilpotent complex Lie algebras of
dimension $\leq 7$ is given.
\end{abstract}

\maketitle
\section{Introduction}

Leibniz algebras
are non-antisymmetric versions
$\mathfrak{g}$
of Lie algebras: the commutator is not required to be antisymmetric,
and the right adjoint operations $[.,Z]$
are required to be derivations
for any $Z \in \mathfrak{g}$
(\cite{loday}).
In the presence of antisymmetry, that is equivalent
to the Jacobi identity.
Leibniz algebras have a cohomology of their own,
the Leibniz cohomology
$HL^\bullet(\mathfrak{g},\mathfrak{g}),$
associated to the complex
$CL^\bullet(\mathfrak{g},\mathfrak{g}) =  \text{Hom }
\left(\mathfrak{g}^{\otimes \bullet}, \mathfrak{g} \right)
= \mathfrak{g} \otimes \left(\mathfrak{g}^*\right)^{\otimes \bullet}$
and the Leibniz coboundary $\delta$ defined for $\psi \in
CL^n(\mathfrak{g},\mathfrak{g})$ by
\begin{multline*}
(\delta \psi)(X_1,X_2, \cdots , X_{n+1})
= \\  [X_1,\psi(X_2,\cdots, X_{n+1})]
 +
 \sum_{i=2}^{n+1} \,(-1)^i [\psi(X_1, \cdots, \hat{X_i}, \cdots, X_{n+1}), X_i]
 \\
 +
 \sum_{1\leqslant i < j \leqslant n+1} (-1)^{j+1}  \,
 \psi(X_1, \cdots, X_{i-1},[X_i,X_j],X_{i+1},\cdots, \hat{X_j}, \cdots, X_{n+1})
\end{multline*}
(If $\mathfrak{g}$ is a Lie algebra, $\delta$ coincides with the usual coboundary $d$
on
$C^\bullet(\mathfrak{g},\mathfrak{g}) = \mathfrak{g} \otimes \bigwedge^{\bullet} \, \mathfrak{g}^*$ ).
Since Lie algebras are Leibniz algebras,
a natural question is,
given some fixed Lie algebra,
whether
or not it has more
infinitesimal Leibniz deformations
(i.e. deformations as a Leibniz algebra)
than infinitesimal deformations as a Lie algebra.
That amounts to the comparison of
the adjoint Leibniz
 2-cohomology group
$HL^2(\mathfrak{g},\mathfrak{g})$
and the ordinary one $H^2(\mathfrak{g},\mathfrak{g}),$
and was addressed by elementary methods
in \cite{jointpaper}. There we proved that
$$HL^2(\mathfrak{g},\mathfrak{g}) = H^2(\mathfrak{g},\mathfrak{g}) \oplus
ZL^2_0(\mathfrak{g},\mathfrak{g})  \oplus  \mathcal{C},$$
where
$ZL^2_0(\mathfrak{g},\mathfrak{g})$
is the space of
symmetric  Leibniz 2-cocycles
and
$\mathcal{C}$
is a space
consisting  of \textit{coupled} Leibniz 2-cocycles, i.e. the nonzero elements
have the property that their
symmetric and antisymmetric
parts are not cocycles.
The Lie algebra $\mathfrak{g}$ is said to be (adjoint)
$ZL^2$-uncoupling
if
$\mathcal{C} =\{0\}$.
That is best understood in terms of the
Koszul map   $\mathcal{I}$
which associates to any invariant bilinear form $B$
on the Lie algebra
$\mathfrak{g}$
the Koszul form $(X,Y,Z) \mapsto I_B(X,Y,Z)=B([X,Y],Z)$
($X,Y,Z  \in \mathfrak{g}$).
Then
$ZL^2_0(\mathfrak{g},\mathfrak{g})
=\mathfrak{c} \otimes \ker{\mathcal{I}}
$
($\mathfrak{c}$ the center of
$\mathfrak{g}$)
and
$\mathcal{C} \cong
\left(\mathfrak{c} \otimes \text{Im\,}{\mathcal{I}}\right) \cap B^3(\mathfrak{g},\mathfrak{g}).$
Hence
$\mathfrak{g}$
is
$ZL^2$-uncoupling
if and only if
$\left(\mathfrak{c}\, \otimes \text{Im\,}{\mathcal{I}}\right) \cap
B^3(\mathfrak{g},\mathfrak{g}) =\{0\}. $
The class of (adjoint)
$ZL^2$-uncoupling  Lie algebras is rather extensive since it contains,
beside the class of zero center Lie algebras,
the class
of Lie algebras having zero Koszul form, which we call
 $\mathcal{I}$-null Lie algebras.

In the present paper, we examine some properties of
the class of
 $\mathcal{I}$-null Lie algebras.
First,  after proving basic properties
 of $\mathcal{I}$-null Lie algebras, we state in Proposition
\ref{prop1}
 a result for
Lie  algebras having a codimension 1 ideal, connecting
 $\mathcal{I}$-nullity of the ideal  and
 $\mathcal{I}$-nullity or  $\mathcal{I}$-exactedness (i.e. the Koszul form is a coboundary)  of the
 Lie algebra itself.
 Several corollaries are given, and fundamental examples are treated in detail.
 We also give a table (Table 1) for all non
 $\mathcal{I}$-null complex Lie algebras of dimension  $\leqslant 7$. This table is a new result.
Then comes the main result of the paper, Theorem
\ref{corollarynilradical},
which states that
any  nilradical
of a Borel subalgebra
of a finite-dimensional  semi-simple
Lie algebra is
$\mathcal{I}$-null.

We also give a list of Kac-Moody types for indecomposable nilpotent Lie algebras of
dimension $\leq 7$ (Table 2).
Again, that result is new.
\par
Throughout the paper, the base field is $\Cmath.$
\section{The Koszul map and $\mathcal{I}$-null  Lie algebras}

Let $\mathfrak{g}$ be any finite dimensional complex Lie algebra.
Recall that a  symmetric bilinear form
$B \in  S^2 \mathfrak{g}^*$
is said to be invariant (see \cite{koszul}), i.e.
$B \in \left(S^2 \mathfrak{g}^*\right) ^{\mathfrak{g}}$
if and only if
$B([Z,X],Y) =-B(X,[Z,Y])  \; \forall X,Y,Z \in
\mathfrak{g}.$
The Koszul map
$\mathcal{I} \, : \,
\left(S^2 \mathfrak{g}^*\right) ^{\mathfrak{g}}
\rightarrow
\left(\bigwedge^3 \mathfrak{g}^*\right)^{\mathfrak{g}}
\subset Z^3(\mathfrak{g},\Cmath)
$
is defined by $\mathcal{I}(B)=  I_B,$ with
$I_B(X,Y,Z)=  B([X,Y],Z) \; \forall X,Y,Z \in  \mathfrak{g}.$
\begin{lemma}
\label{lemmep}
Denote $\mathcal{C}^2 \mathfrak{g}=[\frak{g},\frak{g}].$
The projection
$\pi \; : \; \mathfrak{g} \rightarrow \mathfrak{g}/\mathcal{C}^2 \mathfrak{g}$
induces an isomorphism
$$\varpi \; : \; \ker{\mathcal{I}} \rightarrow
S^2  \left( \mathfrak{g}/\mathcal{C}^2 \mathfrak{g} \right)^*
.$$
\end{lemma}
\begin{proof}
For  $B \in \ker{\mathcal{I}},$
define
$\varpi (B) \in S^2 \left(\mathfrak{g}/\mathcal{C}^2 \mathfrak{g}\right)^*$
by $$\varpi(B)(\pi(X),\pi(Y))= B(X,Y),\; \forall X,Y \in \frak{g}.$$
$\varpi(B)$ is well-defined since for $X,Y,U,V \in \frak{g}$
\begin{eqnarray*}
B(X+[U,V],Y)&=& B(X,Y)+B([U,V],Y)\\
&=& B(X,Y)+I_B(U,V,Y)\\
&=& B(X,Y) \text{ (as $I_B=0$)}.
\end{eqnarray*}
The map $\varpi$ is injective since $\varpi(B)=0$ implies $B(X,Y)=0 \; \forall X,Y \in \frak{g}.$
To prove that it is onto, let
$\bar{B} \in S^2  \left( \mathfrak{g}/\mathcal{C}^2 \mathfrak{g} \right)^*,$ and
let
$B_\pi \in S^2 \mathfrak{g}^*$
defined by $B_\pi(X,Y)=\bar{B}(\pi(X),\pi(Y)).$
Then
$B_\pi ([X,Y],Z) = \bar{B}(\pi([X,Y]),\pi(Z))= \bar{B}(0,\pi(Z)) =0
\linebreak[4]
\; \forall X,Y,Z \in
\frak{g}, $
hence
$B_\pi \in \left(S^2 \mathfrak{g}^*\right) ^{\mathfrak{g}}$
and  $B_\pi \in \ker{\mathcal{I}}.$
Now, $\varpi(B_\pi)=\bar{B}.$
\end{proof}

From Lemma
\ref{lemmep},
$\dim
\left(S^2 \mathfrak{g}^*\right) ^{\mathfrak{g}}
= \frac{\ell(\ell+1)}{2}
+\dim \text{Im\,}{\mathcal{I}},$
where $\ell = \dim   H^1(\mathfrak{g},\Cmath)
  =\dim{\left(   \mathfrak{g}/\mathcal{C}^2 \mathfrak{g} \right)}.$
For reductive
$\mathfrak{g},$
$\dim
\left(S^2 \mathfrak{g}^*\right) ^{\mathfrak{g}}
= \dim H^3(\mathfrak{g},\Cmath)$
(\cite{koszul}).

\begin{definition}
$\mathfrak{g}$ is said to be
 $\mathcal{I}$-null
 (resp.
 $\mathcal{I}$-exact)
if
$\mathcal{I}=0$
 (resp.
$\text{Im\,}{\mathcal{I}} \subset B^3(\frak{g},\Cmath))$.
\end{definition}

$\mathfrak{g}$ is $\mathcal{I}$-null if and only
$\mathcal{C}^2 \mathfrak{g} \subset \ker{B}$
$\forall B \in   \left(S^2 \mathfrak{g}^*\right)^{\mathfrak{g}}.$
It is standard that for any
$B \in  \left(S^2 \mathfrak{g}^*\right)^{\mathfrak{g}},$ there exists
$B_1 \in  \left(S^2 \mathfrak{g}^*\right)^{\mathfrak{g}}$
such that
$\ker{(B+B_1)} \subset \mathcal{C}^2 \mathfrak{g}.$ Hence
$\bigcap_{B \in
\left(S^2 \mathfrak{g}^*\right)^{\mathfrak{g}}}\, \ker{B} \subset \mathcal{C}^2 \mathfrak{g},$
and
$\mathfrak{g}$ is $\mathcal{I}$-null if and only
$\bigcap_{B \in
\left(S^2 \mathfrak{g}^*\right)^{\mathfrak{g}}}\, \ker{B} = \mathcal{C}^2 \mathfrak{g}.$

\begin{lemma}
(i) Any quotient of a  (not necessarily finite dimensional)
$\mathcal{I}$-null Lie algebra is
$\mathcal{I}$-null;
\\ (ii) Any finite direct product of
$\mathcal{I}$-null Lie algebras is
$\mathcal{I}$-null.
\end{lemma}
\begin{proof}
(i)
Let $\frak{g}$ be any
$\mathcal{I}$-null Lie algebra,
 $\frak{h}$ an ideal of
$\frak{g},$
$\bar{\frak{g}}  = \frak{g}/\frak{h},$
\linebreak[4]
$\pi \, : \,
\frak{g}  \rightarrow  \bar{\frak{g}}$ the projection,
and $\bar{B}
\in \left( S^2 {\bar{\mathfrak{g}}}^* \right) ^ {\bar{\frak{g}}}.$
Define $B_\pi
\in S^2 {{\mathfrak{g}}}^*$
by $B_\pi (X,Y) = \bar{B}(\pi(X),\pi(Y)) \, , X,Y \in
\frak{g}.$
Then
$B_\pi ([X,Y],Z) = \bar{B}(\pi([X,Y]),\pi(Z))
= \linebreak[4]  \bar{B}([\pi(X),\pi(Y)],\pi(Z))
= \bar{B}(\pi(X),[\pi(Y),\pi(Z)])
= \bar{B}(\pi(X),\pi([Y,Z]))= \linebreak[4]
B_\pi (X,[Y,Z]) \,
\forall X,Y,Z \in
\frak{g}, $
hence
$B_\pi \in \left(S^2 \mathfrak{g}^*\right) ^{\mathfrak{g}}$
and
$I_{\bar{B}} \circ (\pi \times \pi \times \pi) =  I_{B_\pi} =0$
since
$\frak{g}$
is $\mathcal{I}$-null.
Hence $I_{\bar{B}}=0.$
\\ (ii)
Let
$\frak{g}  = \frak{g}_1 \times \frak{g}_2$
($ \frak{g}_1 , \frak{g}_2$
 $\mathcal{I}$-null) and
$B\in  \left(S^2 \mathfrak{g}^*\right) ^{\mathfrak{g}}.$
As $B(X_1,[Y_2,Z_2])=B([X_1,Y_2],Z_2)=B(0,Z_2)=0\,
\forall X_1 \in  {\frak{g}}_1 ,
Y_2,Z_2 \in  {\frak{g}}_2 ,$
$B$ vanishes on ${\frak{g}}_1 \times {\mathcal{C}}^2{\frak{g}}_2 $
and on $ {\mathcal{C}^2}{\frak{g}}_1 \times  {\frak{g}}_2$ as well,
hence $I_B=0.$
\end{proof}
\begin{lemma}
\label{lemmeborel}
Let $\frak{g}$ be a finite dimensional semi-simple Lie algebra, with Cartan subalgebra
$\frak{h},$
simple root system $S,$ positive roots $\Delta_+,$
and root subspaces ${\frak{g}}^\alpha.$
Let $\frak{k}\neq \{ 0 \}$ be any  subspace of  $\frak{h},$
and $\Gamma \subset
\Delta_+$
such that
$\alpha + \beta \in  \Gamma $
for $\alpha, \beta \in \Gamma, \, \alpha + \beta \in \Delta_+.$
Consider $\frak{u} =\frak{k} \oplus \bigoplus_{\alpha \in \Gamma } {\frak{g}}^\alpha.$
\\ (i) Suppose that $\alpha_{ | \frak{k} } \neq 0 \, \forall \alpha \in \Gamma.$
Then
$\frak{u}$ is $\mathcal{I}$-null;
\\ (ii)
Suppose that $\alpha_{ | \frak{k} } = 0 \, \forall \alpha \in \Gamma \cap S,$
and
$\alpha_{ | \frak{k} } \neq 0 \, \forall \alpha \in \Gamma \setminus S.$
Then
$\frak{u}$ is $\mathcal{I}$-null.
\end{lemma}
\begin{proof}
(i)
Let $\frak{u}_+=  \bigoplus_{\alpha \in \Gamma } {\frak{g}}^\alpha,$
and $X_\alpha$ a root vector in
${\frak{g}}^\alpha$ :
${\frak{g}}^\alpha = \Cmath X_\alpha$ $\forall \alpha \in \Gamma.$
Let  $B\in  \left(S^2 \mathfrak{u}^*\right) ^{\mathfrak{u}}.$
First,
$B(H,X)=0\,  \forall H \in \frak{k},  X
\in \frak{u}_+.$
In fact, for any $\alpha \in \Gamma, $  since
there exists $H_\alpha \in \frak{k} $ such that $\alpha(H_ \alpha) \neq 0,$
$B(H, X_\alpha)=\frac{1}{\alpha(H_\alpha)} \,
B(H, [H_\alpha, X_\alpha])
= \frac{1}{\alpha(H_\alpha)} \,
B([H, H_\alpha], X_\alpha])
= \frac{1}{\alpha(H_\alpha)} \,
B(0, X_\alpha)=0.$
Second, that entails that the restriction of $B$ to
$\frak{u}_+ \times  \frak{u}_+$
is zero, since for any
$\alpha, \beta  \in \Gamma,$
$$B(X_\alpha, X_\beta)= \frac{1}{\alpha(H_\alpha)} B([H_\alpha,X_\alpha], X_\beta)
= \frac{1}{\alpha(H_\alpha)} B(H_\alpha,[X_\alpha, X_\beta])
= 0$$
as $[X_\alpha, X_\beta] \in \frak{u}_+.$
Then
$\frak{u}$ is
$\mathcal{I}$-null.
\\ (ii)
In that case,
$X_\alpha \not \in {\mathcal{C}}^2 \frak{u} \; \forall \alpha \in \Gamma \cap S,$
and $\dim{\left(\frak{u}\left/{\mathcal{C}}^2 \frak{u}\right.\right) }= \dim{\frak{k}}
+\#(\Gamma \cap S).$
For $\frak{u}$ to be
$\mathcal{I}$-null, one has to prove that,
for any  $B\in  \left(S^2 \mathfrak{u}^*\right) ^{\mathfrak{u}}:$
\begin{eqnarray}
\label{1}
B(H,X_\beta)=0 \; \forall H \in \frak{k}, \, \beta \in \Gamma \setminus S ;\\
\label{2}
B(X_\alpha,X_\beta)=0 \; \forall
\alpha \in \Gamma \cap S \, , \,  \beta \in \Gamma \setminus S ;\\
\label{3}
B(X_\beta,X_\gamma)=0 \; \forall
\beta, \gamma \in \Gamma \setminus S.
\end{eqnarray}
(\ref{1}) is proved as in case (i).
To prove (\ref{2}), let $H_\beta \in \frak{k}$ such that $\beta(H_\beta) \neq 0.$
Then
\begin{multline*}
B(X_\alpha,X_\beta)= \frac{1}{\beta(H_\beta)}
B(X_\alpha,[H_\beta,X_\beta])=\frac{1}{\beta(H_\beta)}
B([X_\alpha,H_\beta],X_\beta)= \\
-\frac{1}{\beta(H_\beta)}B(\alpha(H_\beta)
X_\alpha,X_\beta)= -\frac{1}{\beta(H_\beta)}B(0,X_\beta)= 0.
\end{multline*}
As to  (\ref{3}),
\begin{multline*}
B(X_\beta,X_\gamma)=\frac{1}{\beta(H_\beta)}
B([H_\beta,X_\beta],X_\gamma)=\frac{1}{\beta(H_\beta)}
B(H_\beta,[X_\beta,X_\gamma])= 0 \;\text{ from (\ref{1}).}
\end{multline*}

\end{proof}

\begin{example}
Any Borel subalgebra is
$\mathcal{I}$-null.
\end{example}

\begin{proposition}
\label{prop1}
Let $\mathfrak{g}_2$ be a codimension 1 ideal of
the Lie algebra $\mathfrak{g},$
$(x_1, \cdots , x_N)$ a basis of
$\mathfrak{g}$ with $x_1\not \in \mathfrak{g}_2,$
$x_2, \cdots , x_N \in \mathfrak{g}_2,$
$\pi_2 $
 the corresponding projection onto $\mathfrak{g}_2,$
and $(\omega^1, \cdots ,\omega^N)$ denote the dual basis for
$\mathfrak{g}^*.$
Let $B \in   \left(S^2 \mathfrak{g}^*\right)^{\mathfrak{g}},$
and denote
$B_2 \in   \left(S^2 \mathfrak{g}_2^*\right)^{\mathfrak{g}_2}$
the restriction of $B$ to
$\mathfrak{g}_2 \times \mathfrak{g}_2.$
Then:
\\ (i)
\begin{equation}
\label{equationI_Bproposition}
I_B=
d(\omega^1 \wedge f) + I_{B_2}\circ (\pi_2\times \pi_2 \times \pi_2).
\end{equation}
where
$f = B(\cdot , x_1) \in \mathfrak{g}^*;$
\\ (ii)
Let   $\gamma \in \bigwedge^2 \mathfrak{g}_2^* \subset \bigwedge^2 \mathfrak{g}^*,$
and denote $d_{\mathfrak{g}_2}$ the coboundary operator
of ${\mathfrak{g}_2}.$  Then
\begin{equation}
\label{drho}
d\gamma =  \omega^1 \wedge \theta_{x_1} (\gamma)
 +
d_{\mathfrak{g}_2} \gamma
\circ (\pi_2 \times \pi_2 \times \pi_2)
\end{equation}
where $\theta_{x_1}$ stands for the coadjoint action of $x_1$ on the cohomology
of $\mathfrak{g};$
\\ (iii)
Suppose $I_{B_2} \in B^3(\mathfrak{g}_2,\Cmath),$ and
let  $\gamma \in \bigwedge^2 \mathfrak{g}_2^* \subset \bigwedge^2 \mathfrak{g}^*$ such that
$I_{B_2} = d_{\mathfrak{g}_2} \gamma.$
Then
$I_{B} \in B^3(\mathfrak{g},\Cmath)$
if and only if
$\omega^1 \wedge \theta_{x_1} (\gamma)
\in B^3(\mathfrak{g},\Cmath).$
In particular, the condition
\begin{equation}
\label{condoneta}
\theta_{x_1} (\gamma) =df
\end{equation}
implies $I_B=d\gamma.$
\end{proposition}

\begin{proof}
(i)
For $X,Y,Z \in \mathfrak{g}$ one has
\begin{multline}
\label{recIB}
B([X,Y],Z) = B([\omega^1(X)x_1+\pi_2(X), \omega^1(Y)x_1+\pi_2(Y)],\omega^1(Z)x_1+\pi_2(Z))
\\ = B\left(\omega^1(X)[x_1,\pi_2(Y)]-\omega^1(Y)
[x_1,\pi_2(X)] +[\pi_2(X),\pi_2(Y)], \right.\\
\left.
\omega^1(Z)x_1+\pi_2(Z)\right)
\\ = \omega^1(X)\omega^1(Z) B([x_1,\pi_2(Y)],x_1)
-\omega^1(Y)\omega^1(Z) B([x_1,\pi_2(X)],x_1)     \\
+\beta (X,Y,Z) +B\left([\pi_2(X),\pi_2(Y)],\pi_2(Z)\right)
\\
=\beta (X,Y,Z) +B\left([\pi_2(X),\pi_2(Y)],\pi_2(Z)\right)
\end{multline}
where
\begin{multline*}
\beta (X,Y,Z)= \omega^1(Z) B([\pi_2(X),\pi_2(Y)],x_1)
+\omega^1(X) B([x_1,\pi_2(Y)],\pi_2(Z)) \\
-\omega^1(Y) B([x_1,\pi_2(X)],\pi_2(Z))
\\ = \omega^1(Z) B([\pi_2(X),\pi_2(Y)],x_1)
+\omega^1(X) B(x_1,[\pi_2(Y),\pi_2(Z)]) \\
-\omega^1(Y) B(x_1,[\pi_2(X),\pi_2(Z)])
.
\end{multline*}
Now
\begin{eqnarray*}
df(X,Y)&=&-B([X,Y],x_1)\\ &=&
-B([\omega^1(X)x_1+\pi_2(X),\omega^1(Y)x_1+\pi_2(Y)],x_1)
\\
&=&
-B(\omega^1(X)[x_1,\pi_2(Y)]-\omega^1(Y)[x_1,\pi_2(X)]
+[\pi_2(X),\pi_2(Y)],x_1)
\\
&=&
-B([\pi_2(X),\pi_2(Y)],x_1),
\end{eqnarray*}
hence
\begin{eqnarray*}
\beta (X,Y,Z)&=&
-(\omega^1(Z) df(X,Y)
+\omega^1(X) df(Y,Z)
-\omega^1(Y) df(X,Z)) \\&=& -(\omega^1 \wedge df)(X,Y,Z).
\end{eqnarray*}
Since $d\omega^1=0,$
 (\ref{recIB}) then reads
\begin{equation}
I_B=
d(\omega^1 \wedge f) + I_{B_2}\circ (\pi_2\times \pi_2 \times \pi_2).
\end{equation}
\\(ii)
One has for any $X,Y,Z \in  \mathfrak{g}$
\begin{multline*}
d\gamma (X,Y,Z) = d\gamma (\pi_2(X),\pi_2(Y),\pi_2(Z))
+\omega^1(X)d\gamma (x_1,\pi_2(Y),\pi_2(Z))
\\
+\omega^1(Y)d\gamma (\pi_2(X),x_1,\pi_2(Z))
+\omega^1(Z)d\gamma (\pi_2(X),\pi_2(Y),x_1).
\end{multline*}
Now, since $\gamma$ vanishes if one of its arguments is $x_1,$
\begin{eqnarray*}
d\gamma (x_1,\pi_2(Y),\pi_2(Z))
&=&-\gamma ([x_1,\pi_2(Y)],\pi_2(Z)) +\gamma ([x_1,\pi_2(Z)],\pi_2(Y))\\
d\gamma (\pi_2(X),x_1,\pi_2(Z))
&=&-\gamma ([\pi_2(X),x_1],\pi_2(Z)) -\gamma ([x_1,\pi_2(Z)],\pi_2(X))\\
d\gamma (\pi_2(X),\pi_2(Y),x_1)
&=&\gamma ([\pi_2(X),x_1],\pi_2(Y)) -\gamma ([\pi_2(Y),x_1],\pi_2(X)),
\end{eqnarray*}
hence
\begin{multline*}
d\gamma (X,Y,Z) = d\gamma (\pi_2(X),\pi_2(Y),\pi_2(Z))
+\omega^1(X) \theta_{x_1}\gamma (\pi_2(Y),\pi_2(Z))\\
-\omega^1(Y) \theta_{x_1}\gamma (\pi_2(X),\pi_2(Z))
+\omega^1(Z) \theta_{x_1}\gamma (\pi_2(X),\pi_2(Y))
\\
=d\gamma (\pi_2(X),\pi_2(Y),\pi_2(Z))
+\left(\omega^1 \wedge  \theta_{x_1}\gamma \right)(X,Y,Z)
\end{multline*}
since
$\theta_{x_1}\gamma (\pi_2(U),\pi_2(V))
=\theta_{x_1}\gamma (U,V)$
for all $U,V \in  \mathfrak{g}.$
\\(iii) Results immediately from (i) and (ii).
\end{proof}
\par

\begin{corollary}
\label{thecorr}
Under the hypotheses of Proposition
\ref{prop1},
suppose that
$x_1$ commutes with every $x_i$ ($2 \leqslant i \leqslant N$)
except for $x_{i_1}, \cdots, x_{i_r}$  and  that
 $x_{i_1}, \cdots, x_{i_r}$  commute to one another.
Then, if
 $\frak{g}_2$ is $\mathcal{I}$-null,
 $\frak{g}$ is $\mathcal{I}$-null.
\end{corollary}
\begin{proof}
From Equation
\ref{equationI_Bproposition},
one has to prove that
for any invariant bilinear symmetric form $B$
on $\frak{g},$
$f = B(\cdot , x_1) \in \mathfrak{g}^*$ verifies
$df=0,$  i.e.
for any $2\leqslant i,j \leqslant N$,
$B(x_1,[x_i,x_j])=0.$
For $i \neq i_1, \cdots, i_r,$ and any $j\geqslant 2,$
$B(x_1,[x_i,x_j])=B([x_1,x_i],x_j)$
\linebreak[4]
$=B(0,x_j)=0.$
For $i,j \in \{i_1, \cdots, i_r\},$
$B(x_1,[x_i,x_j])= B(x_1,0)=0.$
\end{proof}

\begin{definition}
The $n$-dimensional standard filiform Lie algebra is the Lie algebra
with basis $\{x_1, \cdots, x_n\}$ and commutation relations
$[x_1,x_i]=x_{i+1} \; (1 \leqslant i < n).$
\end{definition}

\begin{corollary}
Any standard filiform Lie algebra or any Heisenberg Lie algebra is
$\mathcal{I}$-null.
\end{corollary}

\par
\begin{corollary}
Any Lie algebra  containing some
$\mathcal{I}$-null
codimension 1 ideal
is $\mathcal{I}$-exact.
\end{corollary}

\begin{corollary}
\label{c11}
Suppose that  the Lie algebra $\mathfrak{g}$ is such that
$\dim{\text{Im\,}{\mathcal{I}}}=0 \text{ or }1. $
Let $\tau \in \text{Der } \frak{g}$ such that
$\tau x_k \in
\mathcal{C}^2 \mathfrak{g}$ $\forall k\geqslant 2$
where $(x_1, \cdots, x_N)$ is some basis of
$\frak{g}.$
Denote $\tilde{\frak{g}}_{\tau} = \Cmath \, \tau \oplus \frak{g}$
the Lie algebra obtained by adjoining the derivation $\tau$ to $\frak{g},$
and by $\tilde{\mathcal{I}}$ the Koszul map of
$\tilde{\frak{g}}_{\tau}.$
Then  $\dim{\text{Im\,}{\tilde{\mathcal{I}}}}=0$ if
$\dim{\text{Im\,}{\mathcal{I}}}=0,$
and   $\dim{\text{Im\,}{\tilde{\mathcal{I}}}}=0 \text{ or }1$
if $\dim{\text{Im\,}{\mathcal{I}}}=1.$

\end{corollary}
\begin{proof}
Let $B \in \left( S^2 {\tilde{\frak{g}}}_\tau {}^*\right)^{\tilde{\frak{g}}_\tau}.$
One has
\begin{equation}
\label{I_Bcorr}
I_B= \omega^{\tau}\wedge df_\tau +I_{B_2} \circ (\pi_2 \times \pi_2 \times \pi_2)
\end{equation}
where $(\tau, x_1, \cdots, x_N)$ is the basis of
$\tilde{\frak{g}}_{\tau},$
$(\omega^\tau, \omega^1, \cdots, \omega^N)$  the dual basis, $B_2$ the restriction
of $B$ to $\frak{g}$, $f_\tau = B(\tau, \cdot \, )$
and $\pi_2$ the projection on $\frak{g}.$
We will also use the projection
$\pi_3$ on $\text{vect}(x_2, \cdots, x_N).$
For $\tilde{X} , \tilde{Y} \in \tilde{\frak{g}}_{\tau},$
 $\tilde{X} = \omega^\tau(\tilde{X}) \tau + X,$
 $\tilde{Y} = \omega^\tau(\tilde{Y}) \tau +Y,$
 $X=\pi_2( \tilde{X}),
 Y=\pi_2( \tilde{Y}),$ so that
 $
 df_\tau( \tilde{X},\tilde{Y})
=  - B(\tau,  [\tilde{X},\tilde{Y}])
= - B(\tau,  [X,Y])
= - B(\tau,[\omega^1(X)x_1+ \pi_3(X),\omega^1(Y)x_1+ \pi_3(Y)])
 = - \omega^1(X)B(\tau,[x_1,\pi_3(Y)])
 $\linebreak[4]$
+ \omega^1(Y)B(\tau,[x_1,\pi_3(X)])
- B(\tau,[\pi_3(X),\pi_3(Y)]), $
hence
\begin{equation}
\label{corrgtau}
 df_\tau( \tilde{X},\tilde{Y})
 =
\omega^1(X)B_2(\tau \pi_3(Y),x_1)
- \omega^1(Y)B_2(\tau \pi_3(X),x_1)
- B_2(\tau \pi_3(X),\pi_3(Y)).
 \end{equation}
 Note that
$\tau \pi_3(X) ,\tau \pi_3(Y) \in \mathcal{C}^2\frak{g}$
by the hypotheses.
Suppose first that
$\frak{g}$ is
$\mathcal{I}$-null.
Then $B_2(\tau \pi_3(Y),x_1),
B_2(\tau \pi_3(X),x_1),
B_2(\tau \pi_3(X),\pi_3(Y))$ all vanish.
From Equations
(\ref{I_Bcorr}),
(\ref{corrgtau}),
$\tilde{\frak{g}}_{\tau}$ is
$\tilde{\mathcal{I}}$-null.
Suppose now that
$\frak{g}$ verifies
$\dim{\text{Im\,}{\mathcal{I}}}=1$ and let
$C \in \left( S^2 {{\frak{g}}} {}^*\right)^{{\frak{g}}}$
with
$I_C \neq 0.$
If  $\tilde{\frak{g}}_{\tau}$ is not
$\tilde{\mathcal{I}}$-null
we may suppose that $I_B  \neq  0.$
There exists $\lambda \in \Cmath$ such that
$I_{B_2} =\lambda I_C.$
Then
$B_2(\tau \pi_3(Y),x_1)=\lambda C(\tau \pi_3(Y),x_1),$
$B_2(\tau \pi_3(X),x_1)
=\lambda C(\tau \pi_3(X),x_1),$
$B_2(\tau \pi_3(X),\pi_3(Y)) =\lambda C(\tau \pi_3(X),\pi_3(Y)).$
It follows from Equations (\ref{I_Bcorr}),
(\ref{corrgtau}),
that   $\dim{\text{Im\,}{\tilde{\mathcal{I}}}}=1.$
\end{proof}

\begin{definition}
A Lie algebra
$\mathfrak{g}$
is said to be  quadratic if there exists  a nondegenerate invariant bilinear form on
$\mathfrak{g}.$
\end{definition}

Clearly, quadratic nonabelian Lie algebras are not $\mathcal{I}$-null.

\begin{example}
This example is an illustration to Corollary \ref{c11}.
The nilpotent Lie algebra
$\mathfrak{g}_{7,2.4}$ has commutation relations
$[x_1,x_2]=x_3,$
$[x_1,x_3]=x_4,$
$[x_1,x_4]=x_5,$
$[x_1,x_5]=x_6,$
$[x_2,x_5]=-x_7,$
$[x_3,x_4]=x_7.$
$\mathfrak{g}_{7,2.4}$ is quadratic and
$\dim{\text{Im\,}{\mathcal{I}}}=1.$
The  elements of $\text{Der } \mathfrak{g}_{7,2.4} (\text{mod } \text{ad} \mathfrak{g}_{7,2.4})$ are
\begin{equation}
\tau=
\begin{pmatrix}
\xi^1_1&0&0&0&0&0&0\\
\xi^2_1&\xi^2_2&0&0&0&0&0\\
0&0&\xi^1_1+\xi^2_2&0&0&0&0\\
0&0&0&2\xi^1_1+\xi^2_2&0&0&0\\
0&\xi^5_2&0&0&3\xi^1_1+\xi^2_2&0&0\\
\xi^6_1&\xi^6_2&\xi^5_2&0&0&4\xi^1_1+\xi^2_2&0\\
\xi^7_1&0&0&0&0&-\xi^1_2&3\xi^1_1+2\xi^2_2
\end{pmatrix}
\end{equation}
$\tau$ is nilpotent if $\xi^1_1=\xi^2_2=0.$
Denote the nilpotent $\tau$ by
$(\xi^2_1;\xi^5_2;\xi^6_1,\xi^6_2;\xi^7_1).$
Now, projectively equivalent derivations $\tau, \tau^{\prime}$ (see \cite{ART})
give isomorphic
$\tilde{\frak{g}}_{\tau},$
$\tilde{\frak{g}}_{\tau^{\prime}}.$
By reduction using projective equivalence, we are reduced to the following cases:
Case 1. $\xi^2_1\neq 0 : (1;\varepsilon;0,\eta;0);$
Case 2. $\xi^2_1= 0 : (0;\varepsilon;0,\eta;\lambda);$
where $\varepsilon, \eta, \lambda= 0, 1.$
In both cases
$\tilde{\frak{g}}_{\tau}$
is $\mathcal{I}$-null,
except when $\tau=0$ in case 2 where
$\tilde{\frak{g}}_{\tau}$
is the direct product
$\Cmath \times \mathfrak{g}_{7,2.4}$ which is quadratic.
Hence any indecomposable
8-dimensional nilpotent Lie algebra containing a subalgebra isomorphic to
$\mathfrak{g}_{7,2.4}$
is $\mathcal{I}$-null, though
$\mathfrak{g}_{7,2.4}$ is quadratic.
That is in line with the fact  that,
from the   double extension method of   \cite{revoy}, \cite{medina},
any indecomposable quadratic solvable Lie algebra
is
a  \textit{double extension} of a quadratic solvable  Lie algebra  by $\Cmath.$

\end{example}

\begin{example}
\rm
Among the 170 (non isomorphic) nilpotent complex Lie algebras of dimension
$\leqslant 7,$
only a few are not
$\mathcal{I}$-null.
Those are listed in
 Table 1
in the classification of \cite{hindawi}, \cite{Magnin1}
(they are all $\mathcal{I}$-exact).
 Table 1 gives for each of them
$\dim \left(S^2 \mathfrak{g}^*\right) ^{\mathfrak{g}},$
 a basis for $ \left(\left(S^2 \mathfrak{g}^*\right) ^{\mathfrak{g}}
\left/ \right. \ker{\mathcal{I}}\right)$
(which in those cases is one-dimensional),
and the corresponding $I_Bs$.
The results in Table 1 are new and have been obtained, first by explicit
computation of all invariant bilinear forms on each one of the 170 Lie algebras with
the computer algebra system \textit{Reduce} and a program similar to
those in \cite{ART},\cite{Magnin1},
and second by hand calculation of $I_B$ for
non $\mathcal{I}$-null Lie algebras.
$\blacksquare$ denotes quadratic Lie algebras;
for $\omega, \pi \in
\mathfrak{g}^*$,
$\odot $
stands for  the symmetric product
 $\omega \odot \pi =  \omega \otimes \pi +
 \pi \otimes  \omega;$
 $\omega^{i,j,k}$ stands for  $ \omega^i\wedge   \omega^j\wedge  \omega^k.$
\end{example}
\begin{table}[h]%
\caption{Non $\mathcal{I}$-null nilpotent complex Lie algebras of dimension $\leqslant 7$.}
\begin{flushleft}
{\fontsize{8}{7.5} \selectfont
\begin{tabular}{||p{2.2cm}@{\;}|@{\;}p{1.8cm}@{}|@{\;}p{5.3cm}@{}|@{\;}p{2.1cm}||}\hline
\hline
\textbf{algebra}&
\textbf{
$\dim \left(S^2 \mathfrak{g}^*\right) ^{\mathfrak{g}}$
}&
 basis for $ \left(S^2 \mathfrak{g}^*\right) ^{\mathfrak{g}}
\left/ \right. \ker{\mathcal{I}}$
&
$I_B$.
\\ \hline

$\mathfrak{g}_{5,4}\, \blacksquare$
&4
&$\omega^1 \odot \omega^5
-\omega^2 \odot \omega^4
+\omega^3 \otimes \omega^3$
&$\omega^{1,2,3}=d\omega^{1,5}$\\ \hline

$\mathfrak{g}_{6,3} \,\blacksquare$
&7
&$\omega^1 \odot \omega^6
-\omega^2 \odot \omega^5
+\omega^3 \odot \omega^4$
&$\omega^{1,2,3}=d\omega^{1,6}$\\ \hline

$\mathfrak{g}_{6,14}$
&4
&$\omega^1 \odot \omega^6
-\omega^2 \odot \omega^4
+\omega^3 \otimes \omega^3$
&$\omega^{1,2,3}=-d\omega^{1,4}$
\\ \hline

$\mathfrak{g}_{5,4} \times \Cmath$ \, $\blacksquare$
&7
&$\omega^1 \odot \omega^5
-\omega^2 \odot \omega^4
+\omega^3 \otimes \omega^3$
&$\omega^{1,2,3}=d\omega^{1,5}$\\ \hline

$\mathfrak{g}_{5,4} \times \Cmath^2$ \, $\blacksquare$
&11
&$\omega^1 \odot \omega^5
-\omega^2 \odot \omega^4
+\omega^3 \otimes \omega^3$
&$\omega^{1,2,3}=d\omega^{1,5}$\\ \hline

$\mathfrak{g}_{6,3} \times \Cmath$ \, $\blacksquare$
&11
&$\omega^1 \odot \omega^6
-\omega^2 \odot \omega^5
+\omega^3 \odot \omega^4
$
&$\omega^{1,2,3}=d\omega^{1,6}$\\ \hline

$\mathfrak{g}_{7,0.4(\lambda)},$
\linebreak[4]
$\mathfrak{g}_{7,0.5},$
$\mathfrak{g}_{7,0.6},$
$\mathfrak{g}_{7,1.02},$
$\mathfrak{g}_{7,1.10},$
$\mathfrak{g}_{7,1.13},$
$\mathfrak{g}_{7,1.14},$
$\mathfrak{g}_{7,1.17}$
&4
&$\omega^1 \odot \omega^5
-\omega^2 \odot \omega^4
+\omega^3 \otimes \omega^3$
&$\omega^{1,2,3}=d\omega^{1,5}$\\ \hline

$\mathfrak{g}_{7,1.03}$
&4
&$\omega^1 \odot \omega^6
-\omega^2 \odot \omega^4
+\omega^3 \otimes \omega^3$
&$\omega^{1,2,3}=d\omega^{1,6}$\\ \hline

$\mathfrak{g}_{7,2.2}$
&7
&$\omega^1 \odot \omega^4
-\omega^2 \odot \omega^6
+\omega^3 \odot \omega^5$
&$\omega^{1,2,3}=d\omega^{1,4}$\\ \hline

$\mathfrak{g}_{7,2.4}\, \blacksquare$
&4
&$\omega^1 \odot \omega^7
+\omega^2 \odot \omega^6
-\omega^3 \odot \omega^5
+\omega^4 \otimes \omega^4$
&$\omega^{1,3,4}-\omega^{1,2,5}=d\omega^{1,7}$\\ \hline

$\mathfrak{g}_{7,2.5}, $
$\mathfrak{g}_{7,2.6},$
$\mathfrak{g}_{7,2.7},$
$\mathfrak{g}_{7,2.8},$
$\mathfrak{g}_{7,2.9},$
&4
&$\omega^1 \odot \omega^5
-\omega^2 \odot \omega^4
+\omega^3 \otimes \omega^3$
&$\omega^{1,2,3}=d\omega^{1,5}$\\ \hline

$\mathfrak{g}_{7,2.18}$
&7
&$\omega^1 \odot \omega^6
-\omega^2 \odot \omega^5
+\omega^4 \otimes \omega^4$
&$\omega^{1,2,4}=d\omega^{1,6}$\\ \hline

$\mathfrak{g}_{7,2.44},$
$\mathfrak{g}_{7,3.6}$
&7
&$\omega^1 \odot \omega^6
-\omega^2 \odot \omega^5
+\omega^3 \odot \omega^4
$
&$\omega^{1,2,3}=d\omega^{1,6}$\\ \hline

$\mathfrak{g}_{7,3.23}$
&7
&$\omega^1 \odot \omega^6
-\omega^2 \odot \omega^5
+\omega^3 \otimes \omega^3
$
&$\omega^{1,2,3}=d\omega^{1,6}$\\ \hline
 \hline
\end{tabular}
}
\end{flushleft}
\end{table}
\begin{remark}
There are nilpotent Lie algebras of higher dimension with
\linebreak[4]
   $ \dim \left(S^2 \mathfrak{g}^*\right) ^{\mathfrak{g}}
\left/ \right. \ker{\mathcal{I}} >1.$
For example, in the case of the $10$ dimensional Lie algebra
$\mathfrak{g}$ with commutation relations
$[x_1,x_2]=x_5,
[x_1,x_3]=x_6,
[x_1,x_4]=x_7,
[x_2,x_3]=x_8,
[x_2,x_4]=x_9,
[x_3,x_4]=x_{10},$
   $ \dim \left(S^2 \mathfrak{g}^*\right) ^{\mathfrak{g}}
\left/ \right. \ker{\mathcal{I}} = 4,$
and in the analogous case
of the 15 dimensional nilpotent
Lie algebra with $5$ generators one has
\linebreak[4]
   $ \dim \left(S^2 \mathfrak{g}^*\right) ^{\mathfrak{g}}
\left/ \right. \ker{\mathcal{I}} = 10.$
Those algebras are  $\mathcal{I}$-exact and not quadratic.
\end{remark}
\begin{example}
The quadratic $5$-dimensional nilpotent Lie algebra
$\mathfrak{g}_{5,4}$
has commutation relations
$[x_1,x_2]=x_3,$
$[x_1,x_3]=x_4,$
$[x_2,x_3]=x_5.$
Consider the $10$-dimensional direct product
$\mathfrak{g}_{5,4} \times \mathfrak{g}_{5,4},$
with the commutation relations:
$[x_1,x_2]=x_5,$
$[x_1,x_5]=x_6,$
$[x_2,x_5]=x_7,$
$[x_3,x_4]=x_8,$
$[x_3,x_8]=x_9,$
$[x_4,x_8]=x_{10}.$
The only $11$-dimensional nilpotent Lie algebra
with an invariant bilinear form which reduces to
$B_1=\omega^1 \odot \omega^7
-\omega^2 \odot \omega^6
+\omega^5 \otimes \omega^5,$
$B_2=\omega^3 \odot \omega^{10}
-\omega^4 \odot \omega^9
+\omega^8 \otimes \omega^8,$
on respectively the first and second factor is
the direct product
$\Cmath \times \mathfrak{g}_{5,4} \times \mathfrak{g}_{5,4},$
\end{example}

\begin{example}
The 4-dimensional solvable "diamond" Lie algebra $\frak{g}$ with basis
$(x_1,x_2,x_3,x_4)$ and commutation relations
$[x_1,x_2]=x_3,
[x_1,x_3]=-x_2,
[x_2,x_3]=x_4$
cannot be obtained as in Lemma
\ref{lemmeborel}.
Here
 $ \dim \left(\left(S^2 \mathfrak{g}^*\right) ^{\mathfrak{g}}
\left/ \right. \ker{\mathcal{I}}\right)=1,$
with basis element
$B= \omega^1 \odot \omega^4  + \omega^2 \otimes \omega^2
+ \omega^3 \otimes \omega^3.$
 $I_B= \omega^{1,2,3} =d\omega^{1,4};$
$\frak{g}$ is quadratic and $\mathcal{I}$-exact.
In fact, one verifies that all other solvable 4-dimensional Lie solvable Lie algebras are
 $\mathcal{I}$-null (for a list, see e.g. \cite{ovando}).
 For a complete description of Leibniz and Lie deformations of the
 diamond Lie algebra
 (and a study of the case of  $\mathfrak{g}_{5,4}$), see
 \cite{jointpaper}.
\end{example}
\par


\section{Case of a nilradical}

We now state and prove our main result.
The proof is by case analysis over
the simple complex finite dimensional Lie algebras.
In the classical cases, the point consists in an inductive use of
 Corollary \ref{thecorr}.
 In the exceptional cases, we either utilize        directly the
 commutation relations ($G_2,F_4$), or make use of a certain property
 of the pattern of positive roots, which we call property $(\mathcal{P})$
 ($E_6,E_7,E_8$).

\begin{theorem}
\label{corollarynilradical}
Any  nilradical $\frak{g}$
of a Borel subalgebra
of a finite-dimensional  semi-simple
Lie algebra is
$\mathcal{I}$-null.
\end{theorem}
\begin{proof}
It is enough to consider the case of a simple Lie algebra, hence of one of the 4 classical types
plus the 5 exceptional ones.
\par \underline{Case $A_n$}.
Denote $E_{i,j}, 1\leqslant i,j \leqslant n+1$ the canonical basis of
$\frak{gl}(n+1,\Cmath).$
One may suppose that the Borel subalgebra of
$A_{n}= \frak{sl}(n+1)$
is comprised of the upper triangular matrices with zero trace, and the Cartan subalgebra $\frak{h}$
is $\bigoplus_{i=1}^{i=n}\; \Cmath H_i$
with
$H_i=E_{i,i}-E_{i+1,i+1}.$
The nilradical is
$\frak{g}=A_n^+=\bigoplus_{1\leqslant i < j \leqslant n+1}\; \Cmath {E}_{i,j}.$
For $n=1$,
 $\mathfrak{g}=\Cmath$ is
$\mathcal{I}$-null.
Suppose the result holds for the nilradical of
the Borel subalgebra of
$A_{n-1}= \frak{sl}(n).$ One has
 $\mathfrak{g}
 =  \Cmath E_{1,2} \oplus  \cdots  \oplus \Cmath E_{1,n+1} \oplus  \mathfrak{g}^{\prime}_2$
 with
 $\mathfrak{g}^{\prime}_2 =\bigoplus_{2\leqslant i < j \leqslant n+1}\; \Cmath {E}_{i,j}$
being the nilradical of
the Borel subalgebra of
$A_{n-1},$
hence $\mathcal{I}$-null.
 $E_{1,n+1}$ commutes with
 $\mathfrak{g}^{\prime}_2,$
 hence
 $ \mathfrak{g}^{\prime}_2$ is a codimension 1 ideal of
 $\Cmath E_{1,n+1} \oplus  \mathfrak{g}^{\prime}_2$,
 and,  from Corollary \ref{thecorr},
 $\Cmath E_{1,n+1} \oplus  \mathfrak{g}^{\prime}_2$ is $\mathcal{I}$-null.
Now $E_{1,n}$ commutes
with  all members of the basis of
$\Cmath E_{1,n+1} \oplus  \mathfrak{g}^{\prime}_2,$
except for $E_{n,n+1},$ and
$[E_{1,n},E_{n,n+1}]=E_{1,n+1}.$
Then
 $\Cmath E_{1,n+1} \oplus  \mathfrak{g}^{\prime}_2$
 is a codimension 1  ideal of
 $\Cmath E_{1,n} \oplus
 ( \Cmath E_{1,n+1} \oplus  \mathfrak{g}^{\prime}_2),$ and
 from Corollary \ref{thecorr},
 $\Cmath E_{1,n} \oplus
 ( \Cmath E_{1,n+1} \oplus  \mathfrak{g}^{\prime}_2)$
is $\mathcal{I}$-null.
Consider
 $\Cmath E_{1,n-1}\oplus (\Cmath E_{1,n}\oplus  \Cmath E_{1,n+1} \oplus  \mathfrak{g}^{\prime}_2).$
 $\Cmath E_{1,n-1}$ commutes
with  all members of the basis of
 $ \Cmath E_{1,n}\oplus  \Cmath E_{1,n+1} \oplus  \mathfrak{g}^{\prime}_2$
except for
$E_{n-1,n},E_{n-1,n+1},$
and yields respectively
$E_{1,n},E_{1,n+1}.$
 Then
 $ \Cmath E_{1,n}\oplus  \Cmath E_{1,n+1} \oplus  \mathfrak{g}^{\prime}_2$
 is a codimension 1  ideal of
 $\Cmath E_{1,n-1}\oplus (\Cmath E_{1,n}\oplus  \Cmath E_{1,n+1} \oplus  \mathfrak{g}^{\prime}_2),$
 and
since
$E_{n-1,n},E_{n-1,n+1}$
commute, we get
 from Corollary \ref{thecorr}
 that $\Cmath E_{1,n-1}\oplus (\Cmath E_{1,n}\oplus  \Cmath E_{1,n+1} \oplus  \mathfrak{g}^{\prime}_2)$
is $\mathcal{I}$-null.
The result then follows by induction.
\par \underline{Case $D_n$}.
We may take $D_n$ as the Lie algebra of matrices
\begin{equation}
\label{Dn}
\begin{pmatrix}
Z_1&Z_2\\Z_3&-{}^tZ_1
\end{pmatrix}
\end{equation}
with $Z_i  \in \frak{gl}(n,\Cmath), Z_2,Z_3$ skew symmetric (see \cite{helgason}, p. 193).
Denote $\tilde{E}_{i,j}=
\left(\begin{smallmatrix}
E_{i,j}&0\\0&-E_{j,i}
\end{smallmatrix}  \right),$
$\tilde{F}_{i,j}=
\left(\begin{smallmatrix}
0&E_{i,j}-E_{j,i}\\0&0
\end{smallmatrix}  \right)$
( $E_{i,j}, 1\leqslant i,j \leqslant n$ the canonical basis of
$\frak{gl}(n,\Cmath)$).
The Cartan subalgebra $\frak{h}$
is $\bigoplus_{i=1}^{i=n}\; \Cmath H_i$
with
$H_i=\tilde{E}_{i,i}$ and the nilradical
of the Borel subalgebra
is
\begin{equation}
\label{0Dn+}
D_n^+=\bigoplus_{1\leqslant i < j \leqslant n}\; \Cmath \tilde{E}_{i,j}
\oplus  \bigoplus_{1\leqslant i < j \leqslant n}\; \Cmath \tilde{F}_{i,j}.
\end{equation}
All $\tilde{F}_{i,j}$'s commute to one another, and one has:
\begin{equation}
\label{commFtilde}
[\tilde{E}_{i,j},\tilde{F}_{k,l}] = \delta_{j,k} \tilde{F}_{i,l} - \delta_{j,l}\tilde{F}_{i,k}.
\end{equation}
We identify $D_{n-1} $ to a subalgebra of $D_n$ by simply taking the first row and
first column of each block to be zero in (\ref{Dn}).
For $n=2$,
$D_2^+=\Cmath^2$ is $\mathcal{I}$-null.
Suppose the result holds true for
$D_{n-1}^+.$ One has
\begin{equation*}
D_n^+=
\Cmath \tilde{E}_{1,2} \oplus
\Cmath \tilde{E}_{1,3} \oplus  \cdots
\Cmath \tilde{E}_{1,n} \oplus
\Cmath \tilde{F}_{1,n} \oplus  \cdots
\oplus \Cmath \tilde{F}_{1,2} \oplus  D_{n-1}^+.
\end{equation*}
Start with $\Cmath \tilde{F}_{1,2} \oplus  D_{n-1}^+.$
From
(\ref{commFtilde}),
$\tilde{F}_{1,2}$ commutes with
all
$\tilde{E}_{i,j}\,( 2 \leqslant i <j \leqslant n$) hence with  $D_{n-1}^+.$
Then
$D_{n-1}^+$  is  a codimension 1 ideal of
$\Cmath \tilde{F}_{1,2} \oplus  D_{n-1}^+$
and
$\Cmath \tilde{F}_{1,2} \oplus  D_{n-1}^+$
is $\mathcal{I}$-null from  Corollary  \ref{thecorr}.
Consider now  $\Cmath \tilde{F}_{1,3} \oplus (\Cmath \tilde{F}_{1,2} \oplus  D_{n-1}^+).$
Again from
(\ref{commFtilde}),
$\tilde{F}_{1,3}$ commutes with
all elements of the basis of
$D_{n-1}^+$ except   $\tilde{E}_{2,3}$ and
$[\tilde{E}_{2,3}, \tilde{F}_{1,3}] = \tilde{F}_{1,2}.$
Then
$\Cmath \tilde{F}_{1,2} \oplus  D_{n-1}^+$
is  a codimension 1 ideal of
$\Cmath \tilde{F}_{1,3} \oplus (\Cmath \tilde{F}_{1,2} \oplus  D_{n-1}^+),$
and  the latter is $\mathcal{I}$-null.
Suppose that
$\Cmath \tilde{F}_{1,s-1} \oplus \cdots  \oplus \Cmath \tilde{F}_{1,2} \oplus  D_{n-1}^+$
is a
codimension 1 ideal of
$\Cmath \tilde{F}_{1,s} \oplus (\Cmath \tilde{F}_{1,s-1} \oplus \cdots  \oplus \Cmath \tilde{F}_{1,2} \oplus  D_{n-1}^+),$
 and that the latter is $\mathcal{I}$-null.
Consider
$\Cmath \tilde{F}_{1,s+1} \oplus (\Cmath \tilde{F}_{1,s} \oplus \cdots  \oplus \Cmath \tilde{F}_{1,2} \oplus  D_{n-1}^+).$
From
(\ref{commFtilde}),
for $2 \leqslant i < j \leqslant n,$
$[\tilde{E}_{i,j},\tilde{F}_{1,s+1}]= \delta_{j,s+1} \tilde{F}_{1,i}$
is nonzero only for $i=2, \cdots, s,$ and $j=s+1,$ and it is then equal to
$\tilde{F}_{1,i}.$
Then first
$\Cmath \tilde{F}_{1,s} \oplus (\Cmath \tilde{F}_{1,s-1} \oplus \cdots  \oplus \Cmath \tilde{F}_{1,2} \oplus  D_{n-1}^+)$
is a codimension 1 ideal of
$\Cmath \tilde{F}_{1,s+1} \oplus (\Cmath \tilde{F}_{1,s} \oplus \cdots  \oplus \Cmath \tilde{F}_{1,2} \oplus  D_{n-1}^+).$
Second,
 the latter
is $\mathcal{I}$-null from
Corollary \ref{thecorr}.
By induction the above property holds for $s=n.$
Consider now
$\Cmath \tilde{E}_{1,n} \oplus (
\Cmath \tilde{F}_{1,n} \oplus  \cdots
\oplus \Cmath \tilde{F}_{1,2} \oplus  D_{n-1}^+).$
One has for
 $2 \leqslant i <j \leqslant n,$
$[\tilde{E}_{1,n},\tilde{E}_{i,j}]=0,$
$[\tilde{E}_{1,n},\tilde{F}_{i,j}]=   -\delta_{n,j} \tilde{F}_{1,i},$
$[\tilde{E}_{1,n},\tilde{F}_{1,j}]= 0.$
Hence
$\Cmath \tilde{F}_{1,n} \oplus  \cdots
\oplus \Cmath \tilde{F}_{1,2} \oplus  D_{n-1}^+$ is an ideal of
$\Cmath \tilde{E}_{1,n} \oplus (
\Cmath \tilde{F}_{1,n} \oplus  \cdots
\oplus \Cmath \tilde{F}_{1,2} \oplus  D_{n-1}^+)$  and the latter is
$\mathcal{I}$-null.
For $2 \leqslant i <j \leqslant n,$
$1 \leqslant k \leqslant n-2,$
\begin{eqnarray*}
[\tilde{E}_{1,n-k},\tilde{E}_{i,j}]&=& \delta_{n-k,i} \tilde{E}_{1,j}, \cr
[\tilde{E}_{1,n-k},\tilde{F}_{i,j}]&=& \delta_{n-k,i} \tilde{F}_{1,j} - \delta_{n-k,j} \tilde{F}_{1,i},\cr
 [  {\tilde{E}}_{1,n-k} , \tilde{E}_{1,n} ]& =&   \delta_{n-k,1} \tilde{E}_{1,n}=0,\cr
[\tilde{E}_{1,n-k},\tilde{F}_{1,j}]&=& \delta_{n-k,1} \tilde{F}_{1,j}=0.
\end{eqnarray*}
$[\tilde{E}_{1,n-1},\tilde{E}_{i,j}]$
is nonzero only for ($i=n-1,j=n$) and then yields
$\tilde{E}_{1,n};$
$[\tilde{E}_{1,n-1},\tilde{F}_{i,j}]$
is nonzero only for ($i=n-1,j=n$)
or for ($i < j=n-1$)
and yields respectively
$ \tilde{F}_{1,n},$ or
$- \tilde{F}_{1,i}.$
 $[  {\tilde{E}}_{1,n-1} , \tilde{E}_{1,n} ]$ and
$[\tilde{E}_{1,n-1},\tilde{F}_{1,j}]$ are zero for $n\geqslant 3.$
Hence, first
$\Cmath \tilde{E}_{1,n} \oplus
\Cmath \tilde{F}_{1,n} \oplus  \cdots
\oplus \Cmath \tilde{F}_{1,2} \oplus  D_{n-1}^+$
is a codimension 1 ideal of
$\Cmath \tilde{E}_{1,n-1} \oplus (
\Cmath \tilde{E}_{1,n} \oplus
\Cmath \tilde{F}_{1,n} \oplus  \cdots
\oplus \Cmath \tilde{F}_{1,2} \oplus  D_{n-1}^+),$
and second  the latter is $\mathcal{I}$-null,
since
$ \tilde{E}_{n-1,n}$
commutes with
$ \tilde{F}_{n-1,n},$
$ \tilde{F}_{i,n-1}.$
Suppose that
$\Cmath \tilde{E}_{1,n-k+1} \oplus \cdots
\oplus \Cmath \tilde{E}_{1,n}
\oplus \Cmath \tilde{F}_{1,n} \oplus  \cdots
\oplus \Cmath \tilde{F}_{1,2} \oplus  D_{n-1}^+$
is a
codimension 1 ideal of
$\Cmath \tilde{E}_{1,n-k} \oplus (
\Cmath \tilde{E}_{1,n-k+1} \oplus \cdots
\oplus \Cmath \tilde{E}_{1,n}
\oplus \Cmath \tilde{F}_{1,n} \oplus  \cdots
\oplus \Cmath \tilde{F}_{1,2} \oplus  D_{n-1}^+)$
and that the latter is
 $\mathcal{I}$-null.
Consider
$\Cmath \tilde{E}_{1,n-k-1} \oplus (
\Cmath \tilde{E}_{1,n-k} \oplus \cdots
\oplus \Cmath \tilde{E}_{1,n}
\oplus \Cmath \tilde{F}_{1,n} \oplus  \cdots
\oplus \Cmath \tilde{F}_{1,2} \oplus  D_{n-1}^+).$
$[\tilde{E}_{1,n-k-1},\tilde{E}_{i,j}]$
is nonzero only for $i=n-k-1$ and yields then
$\tilde{E}_{1,j};$
$[\tilde{E}_{1,n-k-1},\tilde{F}_{i,j}]= \delta_{n-k-1,i} \tilde{F}_{1,j} - \delta_{n-k-1,j} \tilde{F}_{1,i}$
is nonzero only for $i=n-k-1$ or $j=n-k-1$ and yields resp.
$\tilde{F}_{1,j}$ or
$-\tilde{F}_{1,i}.$
Hence
$\Cmath \tilde{E}_{1,n-k} \oplus \cdots
\oplus \Cmath \tilde{E}_{1,n}
\oplus \Cmath \tilde{F}_{1,n} \oplus  \cdots
\oplus \Cmath \tilde{F}_{1,2} \oplus  D_{n-1}^+$
is an ideal of
$\Cmath \tilde{E}_{1,n-k-1} \oplus (
\Cmath \tilde{E}_{1,n-k} \oplus \cdots
\oplus \Cmath \tilde{E}_{1,n}
\oplus \Cmath \tilde{F}_{1,n} \oplus  \cdots
\oplus \Cmath \tilde{F}_{1,2} \oplus  D_{n-1}^+).$
The latter
 is $\mathcal{I}$-null since
$\tilde{E}_{n-k-1,j}$ commutes with both
$\tilde{F}_{n-k-1,j^{\prime}},$
$\tilde{F}_{i,n-k-1}$ ($j^{\prime} \geqslant n-k$).
The result follows by induction.
\par \underline{Case $B_n$}.
We may take $B_n$ ($n\geqslant 2$) as the Lie algebra of matrices
\begin{equation}
\left(
\begin{array}{c|c|c}
0& u&v\\
\hline
-{}^tv&Z_1&Z_2\\
\hline
-{}^tu&Z_3&-{}^tZ_1
\end{array}
\right)
\end{equation}
with $u,v$ complex $(1\times n)$-matrices,
$Z_i  \in \frak{gl}(n,\Cmath), Z_2,Z_3$ skew symmetric, i.e.
\begin{equation}
\left(
\begin{array}{c|ccc}
0& u&&v\\
\hline
-{}^tv&& &\\
&&\boxed{A} &\\
-{}^tu&&
\end{array}
\right)
\end{equation}
with $A \in D_n.$
We identify $A \in D_n$ to the matrix
\begin{equation*}
\left(
\begin{array}{c|ccc}
0& 0&&0\\
\hline
0&& &\\
&&\boxed{A} &\\
0&&
\end{array}
\right)       \in B_n.
\end{equation*}
The Cartan subalgebra of $B_n$ is then simply
that of $D_n.$
$B_n^+$  consists of the matrices
\begin{equation}
\left(
\begin{array}{c|ccc}
0& 0&&v\\
\hline
-{}^tv&& &\\
&&\boxed{A} &\\
0&&
\end{array}
\right)
\end{equation}
with $v$ complex $(1\times n)$-matrix and $A \in D_n^+.$
For $1 \leqslant q \leqslant n, $ let
${v_q}$ the $(1\times n)$-matrix $(0, \cdots,1,\cdots,0)$ ($1$ in $q^{\text{th}}$ position), and
\begin{equation*}
\tilde{v}_q  = \left(
\begin{array}{c|ccc}
0& 0&&v_q\\
\hline
-{}^tv_q&& &\\
&&\boxed{0} &\\
0&&
\end{array}
\right)
\end{equation*}
Hence $B_n^+ = \left( \bigoplus_{q=1}^{n}
\, \Cmath \tilde{v}_q \right) \oplus D_n^+.$
One has for
$1 \leqslant q \leqslant n, \; 1 \leqslant i<j \leqslant n $
\begin{eqnarray*}
[\tilde{v}_q ,\tilde{E}_{i,j}]&=&   -\delta_{q,j} \tilde{v}_i
\cr
[\tilde{v}_q ,\tilde{F}_{i,j}]&=&   0
\end{eqnarray*}
and for $1 \leqslant s<q \leqslant n$
\begin{equation}
\label{vt}
[\tilde{v}_q ,\tilde{v}_{s}]=   \tilde{F}_{s,q} .
\end{equation}
Consider
$\Cmath \tilde{v}_1 \oplus D_n^+.$
As
$\tilde{v}_1$ commutes with
$\tilde{E}_{i,j}$ and $\tilde{F}_{i,j},$
$D_n^+$ is an ideal of
$\Cmath \tilde{v}_1 \oplus D_n^+$
and the latter is  $\mathcal{I}$-null.
 Suppose that
$\Cmath \tilde{v}_{s-1} \oplus \cdots \oplus \Cmath \tilde{v}_1 \oplus D_n^+$
is an ideal of
\linebreak[4]
$\Cmath \tilde{v}_{s} \oplus \left( \Cmath \tilde{v}_{s-1} \oplus \cdots \oplus \Cmath \tilde{v}_1 \oplus D_n^+\right)$
and the latter is  $\mathcal{I}$-null.
Consider
\linebreak[4]
$\Cmath \tilde{v}_{s+1} \oplus \left( \Cmath \tilde{v}_{s} \oplus  \Cmath \tilde{v}_{s-1} \oplus \cdots \oplus \Cmath \tilde{v}_1 \oplus D_n^+\right).$
$[\tilde{v}_{s+1} ,\tilde{E}_{i,j}]=   -\delta_{s+1,j} \tilde{v}_i$ hence
$\tilde{v}_{s+1}$  commutes to all $\tilde{E}_{i,j}$'s except for
$\tilde{E}_{i,s+1}$ ($i\leqslant s$) and then yields
$-\tilde{v}_i.$
For $t \leqslant s,$
$[\tilde{v}_{s+1} ,\tilde{v}_{t}]=   \tilde{F}_{t,s+1} .$
Hence $\Cmath \tilde{v}_{s} \oplus  \Cmath \tilde{v}_{s-1} \oplus \cdots \oplus \Cmath \tilde{v}_1 \oplus D_n^+$
is an ideal of
$\Cmath \tilde{v}_{s+1} \oplus \left( \Cmath \tilde{v}_{s} \oplus  \Cmath \tilde{v}_{s-1} \oplus \cdots \oplus \Cmath \tilde{v}_1 \oplus D_n^+\right).$
Now we cannot apply directly
 Corollary \ref{thecorr} to conclude
that the latter is  $\mathcal{I}$-null
as the family
$\mathcal{F}=\{ \tilde{E}_{i,s+1}, \tilde{v}_t ;
1\leqslant i \leqslant s,
1\leqslant t \leqslant s
\}$
is not commutative.
The $\tilde{E}_{i,s+1}$'s ($i\leqslant s$) commute to one another and
to the
$\tilde{v}_t$'s, but
the $\tilde{v}_t$'s do not commute to one another.
However, recall from the proof of
 Corollary \ref{thecorr}
 that one has to check that, for any invariant bilinear form $B$ on
$\Cmath \tilde{v}_{s+1} \oplus \left( \Cmath \tilde{v}_{s} \oplus  \Cmath \tilde{v}_{s-1} \oplus \cdots \oplus \Cmath \tilde{v}_1 \oplus D_n^+\right)$,
$B(\tilde{v}_{s+1}, [X,Y])=0$
for all $X,Y \in \mathcal{F}.$
That reduces to
$B(\tilde{v}_{s+1},
[\tilde{v}_{t},\tilde{v}_{t'}])=0$ $\forall t,t', 1\leqslant t <t'\leqslant s.$
Now,
$B(\tilde{v}_{s+1},  [\tilde{v}_{t},\tilde{v}_{t'}])
=
B([\tilde{v}_{s+1},  \tilde{v}_{t}],\tilde{v}_{t'})
= B(\tilde{F}_{t,s+1},\tilde{v}_{t'})
= B([\tilde{E}_{t,s},\tilde{F}_{s,s+1}],\tilde{v}_{t'})
=0$
since
$\tilde{E}_{t,s},\tilde{F}_{s,s+1},\tilde{v}_{t'} \in
\Cmath \tilde{v}_{s} \oplus  \Cmath \tilde{v}_{s-1} \oplus \cdots \oplus \Cmath \tilde{v}_1 \oplus D_n^+$
which is  $\mathcal{I}$-null.
We conclude that $\Cmath \tilde{v}_{s+1} \oplus \left( \Cmath \tilde{v}_{s} \oplus  \Cmath \tilde{v}_{s-1} \oplus \cdots \oplus \Cmath \tilde{v}_1 \oplus D_n^+\right)$
is  $\mathcal{I}$-null.
By induction the property holds for $s=n$ and $B_n^+$
is  $\mathcal{I}$-null.
\par \underline{Case $C_n$}.
This case is pretty similar to the case $D_n.$
We may take $C_n$ as the Lie algebra of matrices
\begin{equation}
\label{Cn}
\begin{pmatrix}
Z_1&Z_2\\Z_3&-{}^tZ_1
\end{pmatrix}
\end{equation}
with $Z_i  \in \frak{gl}(n,\Cmath), Z_2,Z_3$ symmetric.
 $\tilde{E}_{i,j}$ and the Cartan subalgebra are identical to those of $D_n.$
We denote for $1 \leqslant i,j \leqslant n:$
$\hat{F}_{i,j}=  \left(\begin{smallmatrix}
0&E_{i,j}+E_{j,i}\\0&0
\end{smallmatrix}  \right).$ Then
\begin{equation}
\label{Cn+}
C_n^+= \bigoplus_{1\leqslant i < j \leqslant n}\; \Cmath \tilde{E}_{i,j}
\oplus   \bigoplus_{1\leqslant i \leqslant j \leqslant n}\; \Cmath \hat{F}_{i,j}.
\end{equation}
All $\hat{F}_{k,l}$'s commute to one another, and one has:
\begin{equation}
\label{hatF}
[\tilde{E}_{i,j},\hat{F}_{k,l}] = \delta_{j,k} \hat{F}_{i,l} + \delta_{j,l}\hat{F}_{i,k}.
\end{equation}
The case is step by step analogous to the case of $D_n$ with
(\ref{hatF}) instead of (\ref{commFtilde})
and (\ref{Cn+}) instead of (\ref{0Dn+}).
\par \underline{Case $G_2$}.
The commutation relations for $G_2$  appear in \cite{fulton-harris}, p. 346.
$G_2^+$ is 6-dimensional with commutation relations
$[x_1,x_2]=x_3; \,
[x_1,x_3]=2x_4;   \,
[x_1,x_4]=-3x_5;   \,
[x_2,x_5]=-x_6;   \,
[x_3,x_4]=-3x_6.$
$G_2^+$ has the same adjoint cohomology
\linebreak[4]
$(1,4,7,8,7,5,2)$ as,
and is  isomorphic to,
$\mathfrak{g}_{6,18},$ which is
 $\mathcal{I}$-null.
\par \underline{Case $F_4$}.
$F_4^+$ has  24 positive roots,
and root vectors $x_i$ ($1 \leqslant i \leqslant 24$).
From the root pattern, one gets with some calculations the
commutation relations of $F_4^+:$
$ [x_1,x_2]=x_{5};$
$ [x_1,x_{13}]=x_{14};$
$ [x_1,x_{15}]= - x_{6};$
$ [x_1,x_{16}]= - x_{7};$
$ [x_1,x_{17}]= - x_{23};$
$ [x_1,x_{18}]=x_{19};$
$ [x_1,x_{24}]=x_{22};$
$ [x_2,x_3]=x_{15};$
$ [x_2,x_7]=x_{8};$
$ [x_2,x_{12}]=x_{13};$
$ [x_2,x_{19}]=x_{20};$
$ [x_2,x_{21}]=x_{24};$
$ [x_2,x_{23}]=x_{9};$
$ [x_3,x_4]=x_{21};$
$ [x_3,x_5]=x_{6};$
$ [x_3,x_6]=x_{7};$
$ [x_3,x_9]=x_{10};$
$ [x_3,x_{11}]=x_{12};$
$ [x_3,x_{15}]=x_{16};$
$ [x_3,x_{20}]= - 2 x_{11};$
$ [x_3,x_{22}]= \frac{1}{2}  x_{23};$
$ [x_3,x_{24}]= -\frac{1}{2} x_{17};$
$ [x_4,x_6]=x_{22};$
$ [x_4,x_7]=x_{23};$
$ [x_4,x_8]=x_{9};$
$ [x_4,x_9]= - x_{20};$
$ [x_4,x_{10}]=x_{11};$
$ [x_4,x_{15}]= - x_{24};$
$ [x_4,x_{16}]=x_{17};$
$ [x_4,x_{17}]=x_{18};$
$ [x_4,x_{23}]= - x_{19};$
$ [x_5,x_{12}]=x_{14};$
$ [x_5,x_{16}]=x_{8};$
$ [x_5,x_{17}]=x_{9};$
$ [x_5,x_{18}]= - x_{20};$
$ [x_5,x_{21}]=x_{22};$
$ [x_6,x_{11}]= - x_{14};$
$ [x_6,x_{15}]= - x_{8};$
$ [x_6,x_{17}]=x_{10};$
$ [x_6,x_{18}]=2 x_{11};$
$ [x_6,x_{21}]= \frac{1}{2}  x_{23};$
$ [x_6,x_{24}]= \frac{1}{2} x_{9};$
$ [x_7,x_{18}]=2 x_{12};$
$ [x_7,x_{20}]= - 2 x_{14};$
$ [x_7,x_{24}]=x_{10};$
$ [x_8,x_{18}]=2 x_{13};$
$ [x_8,x_{19}]=2 x_{14};$
$ [x_8,x_{21}]= - x_{10};$
$ [x_9,x_{17}]= - 2 x_{13};$
$ [x_9,x_{21}]= - x_{11};$
$ [x_9,x_{23}]=2 x_{14};$
$ [x_{10},x_{21}]= - x_{12};$
$ [x_{10},x_{22}]= - x_{14};$
$ [x_{10},x_{24}]= - x_{13};$
$ [x_{11},x_{15}]= - x_{13};$
$ [x_{15},x_{19}]=2 x_{11};$
$ [x_{15},x_{21}]= \frac{1}{2} x_{17};$
$ [x_{15},x_{22}]= \frac{1}{2}  x_{9};$
$ [x_{15},x_{23}]= - x_{10};$
$ [x_{16},x_{19}]=2 x_{12};$
$ [x_{16},x_{20}]=2 x_{13};$
$ [x_{16},x_{22}]=x_{10};$
$ [x_{17},x_{22}]=x_{11};$
$ [x_{17},x_{23}]=2 x_{12};$
$ [x_{21},x_{22}]=  \frac{1}{2}  x_{19};$
$ [x_{21},x_{24}]=  \frac{1}{2}  x_{18};$
$ [x_{22},x_{24}]=   - \frac{1}{2}   x_{20};$
$ [x_{23},x_{24}]=x_{11}.$

Then the
computation of all invariant bilinear forms on $F_4^+$ with
the computer algebra system \textit{Reduce}
yields the conclusion
that $F_4^+$
 is $\mathcal{I}$-null.
\par \underline{Case $E_6$}.
In the case of $E_6^+$ the set $\Delta_+$ of positive roots
(associated to
the set $S$ of simple roots)
has cardinality
36  (\cite{fulton-harris}, p. 333):
\begin{multline*}
\Delta_+ = \{\varepsilon_i+\varepsilon_j;
1\leqslant i <j \leqslant5\} \cup
\{\varepsilon_i-\varepsilon_j; 1\leqslant j <i \leqslant5\}
\\
\cup
\{\frac{1}{2}(
\pm\varepsilon_1\pm\varepsilon_2
\pm\varepsilon_3\pm\varepsilon_4
\pm\varepsilon_5+\sqrt{3}\, \varepsilon_6);
\text{\# minus signs even}
\}
\end{multline*}
(the $(\varepsilon_j)$'s an orthogonal basis of the Euclidean space).
Instead of computing the commutation relations,
we will use the following
property    ($\mathcal{P}$) of $\Delta_+.$
\begin{equation*}
 (\mathcal{P}):
\text{ for } \alpha, \beta, \gamma \in \Delta_+ \, ,
\text{ if }
\alpha+\beta \in \Delta_+
\text{ and } \alpha+\gamma \in \Delta_+ \,  ,
\text{ then }
\beta+\gamma \not\in \Delta_+ .
\end{equation*}
Introduce some Chevalley basis (\cite{helgason}, p. 19 ex. 7)
of
$E_6^+:$
$(X_\alpha)_{\alpha \in \Delta^+}.$
One has
$$[X_\alpha, X_\beta] =N_{\alpha,\beta} X_{\alpha+\beta}  \; \forall \alpha, \beta \in
\Delta_+$$
$$N_{\alpha,\beta}=0 \text{ if } \alpha+\beta \not \in \Delta_+,
N_{\alpha,\beta}\in \Zmath \setminus \{0\}  \text{ if } \alpha+\beta  \in \Delta_+.$$
Define inductively a sequence
$\frak{g}_1 \subset
\frak{g}_2 \subset \cdots
\subset
\frak{g}_{36} = E_6^+$
of
 $\mathcal{I}$-null subalgebras, each of which a codimension 1 ideal of the following,
 as follows.
 Start with
$\frak{g}_1 = \Cmath X_{\delta_1},$
$\delta_1 \in \Delta_+$ of maximum height.
Suppose
$\frak{g}_i$ defined. Then take
$\frak{g}_{i+1} = \Cmath X_{\delta_{i+1}} \oplus \frak{g}_i$
with $\delta_{i+1} \in \Delta_+ \setminus \{\delta_1, \cdots ,\delta_i\}$
of maximum height.
Clearly,
$\frak{g}_i$ is a codimension 1 ideal of $\frak{g}_{i+1}.$
To prove that it is
 $\mathcal{I}$-null
 we only have to check that, for $1 \leqslant s,t \leqslant i,$
 if
 $\delta_{i+1}+\delta_s \in \Delta_+$
 and
 $\delta_{i+1}+\delta_t \in \Delta_+$
 then $\delta_s +\delta_t \not \in \Delta_+.$
 That holds true because of property  ($\mathcal{P}$).
\par \underline{Case $E_7$}.
In the case of $E_7^+$ the set $\Delta_+$ of positive roots
has cardinality
63  (\cite{fulton-harris}, p. 333):
\begin{multline*}
\Delta_+ = \{\varepsilon_i+\varepsilon_j;
1\leqslant i <j \leqslant6\} \cup
\{\varepsilon_i-\varepsilon_j; 1\leqslant j <i \leqslant6\}
\cup
\{\sqrt{2}\varepsilon_7
\}
\\
\cup
\{\frac{1}{2}(
\pm\varepsilon_1\pm\varepsilon_2
\pm\varepsilon_3\pm\varepsilon_4
\pm\varepsilon_5
\pm\varepsilon_6
+\sqrt{2}\, \varepsilon_7);
\text{\# minus signs odd}
\}.
\end{multline*}
 Property ($\mathcal{P}$) holds true for $E_7^+$
(see \cite{companionarchive}).
 Hence the conclusion follows as in the case of $E_6^+.$
\par \underline{Case $E_8$}.
In the case of $E_8^+$ the set $\Delta_+$ of positive roots
has cardinality
120 (\cite{fulton-harris}, p. 333):
\begin{multline*}
\Delta_+ =  \{\varepsilon_i+\varepsilon_j;
1\leqslant i <j \leqslant8\} \cup
\{\varepsilon_i-\varepsilon_j; 1\leqslant j <i \leqslant8\}
\\
\cup
\{\frac{1}{2}(
\pm\varepsilon_1\pm\varepsilon_2
\pm\varepsilon_3\pm\varepsilon_4
\pm\varepsilon_5
\pm\varepsilon_6
\pm\varepsilon_7
+ \varepsilon_8);
\text{\# minus signs even}
\}.
\end{multline*}
 Property ($\mathcal{P}$) holds true for $E_8^+$
(see \cite{companionarchive}).
 Hence the conclusion follows as in the case of $E_6^+.$
\end{proof}
\begin{remark}
 Property ($\mathcal{P}$) holds  for $A_n^+,$ hence we could have used it.
 However, it does not hold  for $F_4^+.$
 One has for example in the above commutation relations of $F_4^+$
 (with root vectors)
 $[x_3,x_4] \neq 0, [x_3,x_9]\neq 0, $ yet $[x_4,x_9] \neq 0.$
\end{remark}

\begin{remark}
In the \textit{transversal to dimension} approach to the classification problem of nilpotent
Lie algebras initiated in \cite{santha1}, one first associates a generalized Cartan matrix
(abbr. GCM) $A$ to any nilpotent finite dimensional complex Lie algebra $\frak{g},$
and then looks at $\frak{g}$ as the quotient $\hat{\frak{g}}(A)_+/\frak{I}$
of the nilradical
of the Borel subalgebra
of the Kac-Moody Lie algebra
$\hat{\frak{g}}(A)$ associated to $A$ by some ideal $\frak{I}.$
Then one gets for any GCM $A$ the subproblem of classifying
(up to the action of a certain group)
all ideals of
$\hat{\frak{g}}(A)_+,$
thus getting all nilpotent Lie algebras of type $A$
(see
\cite{favre-santha},
\cite{fernandez},
\cite{fernandez-nunez},
\cite{santha4}, and the references therein).
Any indecomposable GCM is of exactly one of the 3 types \textit{finite}, \textit{affine},
\textit{indefinite} (among that last  the hyperbolic GCMs, with the property that
any connected proper subdiagram of the Dynkin diagram is of finite or affine type)
(\cite{carbone},\cite{kac},\cite{wanze}).
From Theorem \ref{corollarynilradical},
the nilpotent Lie algebras that are not $\mathcal{I}$-null
all come from affine or indefinite types.
Unfortunately, that is the case of many nilpotent Lie algebras,
see Table 2.
Finally, let us add some indications on how Table 2 was computed.
The commutation relations for the nilpotent Lie algebras
$\mathfrak{g}$
in Table 2
are given in \cite{ART}, \cite{Magnin1}
in terms of a basis
$(x_j)_{1 \leqslant j \leqslant n}$ ($n =\dim{\mathfrak{g}}$)
which diagonalizes a maximal torus $T$.
We may suppose here that
$(x_j)_{1 \leqslant j \leqslant \ell}$,
 $\ell =\dim{\left(   \mathfrak{g}/\mathcal{C}^2 \mathfrak{g} \right)}$,
 is a basis for
$\mathfrak{g}$
 modulo $\mathcal{C}^2 \mathfrak{g}.$
The associated weight pattern $R(T)$ and weight spaces decomposition
$\mathfrak{g} = \bigoplus_{\beta \in R(T)} \, \mathfrak{g}^\beta$
appear in \cite{Magnin1}. As in
\cite{santha1}, one first introduces
$R_1(T)=\{\beta \in R(T); \mathfrak{g}^{\beta} \not \subset \mathcal{C}^2 \mathfrak{g}\} =\{\beta_1, \dots ,\beta_s\},$
$\ell_a =
 \dim{\left(   \mathfrak{g}^{\beta_a}/
            \left( \mathfrak{g}^{\beta_a}\cap \mathcal{C}^2 \mathfrak{g}
                                       \right)
                                       \right) }$,
 $d_a =  \dim{  \mathfrak{g}^{\beta_a}}$
 ($1\leqslant a \leqslant s$).
By definition the GCM associated to
$\mathfrak{g}$ is
$A=(a^i_j)_{1 \leqslant i,j \leqslant \ell}$
with $a^i_i=2$
and, for $i\neq j,$ $-a^i_j$ defined as follows.
 In the simplest case where $d_a= 1 \, \forall a$
 ($1\leqslant a \leqslant s$),
 then,
for $i\neq j,$ $-a^i_j$ is the lowest  $k \in \Nmath$ such that
$ad(x_i)^{k+1}(x_j)=0.$
If
  $d_a > 1$ for some
 $1\leqslant a \leqslant s$
 (Lie algebras having that property are signalled by a ${ }^{\ddagger}$ in Table 2),
one has (if   $l_a > 1$  as well)
to reorder $x_1, \dots ,x_\ell$
according to weights as
$y_1, \dots ,y_\ell$  with $y_j$ of weight $\beta_{f(j)}$ ,
$f \, : \, \{1,\dots, s\} \rightarrow \{1,\dots, s\}$ some step function.
Then, for $i\neq j$,
$-a^i_j = \inf{\{k \in \Nmath;
ad(v)^{k+1}(w)=0 \,
\forall v \in \mathfrak{g}^{\beta_{f(i)}}
\, \forall w \in \mathfrak{g}^{\beta_{f(j)}}
\} }$.
The GCM $A$ is an invariant of
$\mathfrak{g},$ up to permutations of $\{\beta_1,\dots , \beta_s\}$
that leave the $d_{\beta}$'s invariant.
The type of the GCM was identified either directly or through
the associated Dynkin diagram.
As an example to Table 2, there are (up to isomorphism) three
$7$-dimensional nilpotent Lie algebras
that can be  constructed from the GCM $D^{(3)}_4$:
$\mathfrak{g}_{7,2.1(ii)},$
$\mathfrak{g}_{7,2.10},$
$\mathfrak{g}_{7,3.2}.$
The $7$-dimensional nilpotent Lie algebra
$D^{(3),0}_{4,42}$ constructed from the GCM $D^{(3)}_4$ in
\cite{fernandez} is isomorphic to
$\mathfrak{g}_{7,3.2}.$

\end{remark}

\begin{table}[h]%
\label{KMtable1}
\caption{Kac-Moody types for indecomposable nilpotent Lie algebras of dimension $\leq 7.$
Notations for indefinite hyperbolic are those of
\cite{wanze},
supplemented
in parentheses
for rank $3,4$
by the notations of \cite{carbone}
(as there are misprints and omissions in \cite{wanze}).
}
\begin{flushleft}
{\fontsize{8}{7.5} \selectfont
\begin{tabular}{||p{1.4cm}|p{2cm}|p{1.5cm}|p{1.5cm}|p{1.5cm}|p{1.5cm}||}\hline
\hline
\textbf{algebra}&
\textbf{GCM}&
Finite&
Affine&
Indefinite Hyperbolic&
Indefinite Not Hyperbolic
\\ \hline

$\mathfrak{g}_{3}$
&$\left( \begin{smallmatrix}
2&-1\\-1&2
\end{smallmatrix}\right)$
&         $A_2$             
&                      
&                     
&                      
\\

$\mathfrak{g}_{4}$
&$\left( \begin{smallmatrix}
2&-2\\-1&2
\end{smallmatrix}\right)$
&         $C_2$             
&                      
&                     
&                      
\\

$\mathfrak{g}_{5,1}$
&$\left( \begin{smallmatrix}
2&0&-1&0\\0&2&0&-1\\-1&0&2&0\\0&-1&0&2
\end{smallmatrix}\right)$
&         $A_2 \times A_2$             
&                      
&                     
&                      
\\

$\mathfrak{g}_{5,2}$
&$\left( \begin{smallmatrix}
2&-1&-1\\-1&2&0\\-1&0&2
\end{smallmatrix}\right)$
&         $A_3$             
&                      
&                     
&                      
\\

$\mathfrak{g}_{5,3}$
&$\left( \begin{smallmatrix}
2&-2&0\\-1&2&-1\\0&-1&2
\end{smallmatrix}\right)$
&         $B_3$             
&                      
&                     
&                      
\\

$\mathfrak{g}_{5,4}$
&$\left( \begin{smallmatrix}
2&-2\\-2&2
\end{smallmatrix}\right)$
&                      
&         $A_1^{(1)}$             
&                     
&                      
\\

$\mathfrak{g}_{5,5}$
&$\left( \begin{smallmatrix}
2&-3\\-1&2
\end{smallmatrix}\right)$
&         $G_2$             
&                      
&                     
&                      
\\

$\mathfrak{g}_{5,6}$
&$\left( \begin{smallmatrix}
2&-3\\-2&2
\end{smallmatrix}\right)$
&                      
&                      
&         $(3,2)$            
&                      
\\

$\mathfrak{g}_{6,1}$
&$\left( \begin{smallmatrix}
2&-1&0&-1\\-1&2&-1&0\\0&-1&2&0\\-1&0&0&2
\end{smallmatrix}\right)$
&         $A_4$             
&                      
&                     
&                      
\\

$\mathfrak{g}_{6,2}$
&$\left( \begin{smallmatrix}
2&-2&0&0\\-1&2&0&0\\0&0&2&-1\\0&0&-1&2
\end{smallmatrix}\right)$
&         $B_2 \times A_2$             
&                      
&                     
&                      
\\

$\mathfrak{g}_{6,3}$
&$\left( \begin{smallmatrix}
2&-1&-1\\-1&2&-1\\-1&-1&2
\end{smallmatrix}\right)$
&                     
&        $A_2^{(1)}$              
&                     
&                      
\\

$\mathfrak{g}_{6,4}$
&$\left( \begin{smallmatrix}
2&-1&-1\\-2&2&0\\-1&0&2
\end{smallmatrix}\right)$
&       $B_3$              
&                      
&                     
&                      
\\

$\mathfrak{g}_{6,5}$
${ }^{\ddagger}$
&$\left( \begin{smallmatrix}
2&-2&-1\\-2&2&-1\\-1&-1&2
\end{smallmatrix}\right)$
&                     
&                      
&       $H_2^{(3)}$ \; (32)             
&                      
\\

$\mathfrak{g}_{6,6}$
&$\left( \begin{smallmatrix}
2&-1&0\\-2&2&-1\\0&-1&2
\end{smallmatrix}\right)$
&       $C_3$              
&                      
&                     
&                      
\\

$\mathfrak{g}_{6,7}$
&$\left( \begin{smallmatrix}
2&-2&-1\\-1&2&-1\\-1&-1&2
\end{smallmatrix}\right)$
&                     
&                      
&       $H_1^{(3)}$  \; (1)            
&                      
\\

$\mathfrak{g}_{6,8}$
&$\left( \begin{smallmatrix}
2&-2&0\\-2&2&-1\\0&-1&2
\end{smallmatrix}\right)$
&                     
&                      
&       $H_{96}^{(3)}$ \, (103)             
&                      
\\

$\mathfrak{g}_{6,9}$
&$\left( \begin{smallmatrix}
2&-1&-1\\-1&2&0\\-1&0&2
\end{smallmatrix}\right)$
&       $A_3$              
&                      
&                     
&                      
\\

$\mathfrak{g}_{6,10}$
&$\left( \begin{smallmatrix}
2&-2&-1\\-1&2&0\\-2&0&2
\end{smallmatrix}\right)$
&                     
&       $A_4^{(2)}$               
&                     
&                      
\\

$\mathfrak{g}_{6,11}$
&$\left( \begin{smallmatrix}
2&-3&0\\-1&2&-1\\0&-1&2
\end{smallmatrix}\right)$
&                     
&       $G_2^{(1)}$               
&                     
&                      
\\
$\mathfrak{g}_{6,12}
{ }^{\ddagger}$
&$\left( \begin{smallmatrix}
2&-3&-2\\-2&2&-1\\-1&-1&2
\end{smallmatrix}\right)$
&                     
&                      
&                     
&      $\surd$                
\\

$\mathfrak{g}_{6,13}$
&$\left( \begin{smallmatrix}
2&-3&0\\-1&2&-1\\0&-1&2
\end{smallmatrix}\right)$
&                     
&       $G_2^{(1)}$               
&                     
&                      
\\

$\mathfrak{g}_{6,14}$
&$\left( \begin{smallmatrix}
2&-3\\-2&2
\end{smallmatrix}\right)$
&                     
&                      
&       $(3,2)$              
&                      
\\

$\mathfrak{g}_{6,15}$
&$\left( \begin{smallmatrix}
2&-2\\-2&2
\end{smallmatrix}\right)$
&                     
&       $A_1^{(1)}$               
&                     
&                      
\\

$\mathfrak{g}_{6,16}$
&$\left( \begin{smallmatrix}
2&-4\\-1&2
\end{smallmatrix}\right)$
&                     
&       $A_2^{(2)}$               
&                     
&                      
\\
$\mathfrak{g}_{6,17}$
&$\left( \begin{smallmatrix}
2&-4\\-2&2
\end{smallmatrix}\right)$
&                     
&                      
&       $(4,2)$              
&                      
\\
\hline
\hline

\end{tabular}
}
\end{flushleft}
\end{table}

\setcounter{table}{1}
\begin{table}[h]%
\caption{continued}
\begin{flushleft}
{\fontsize{8}{7.5} \selectfont
\begin{tabular}{||p{1.4cm}|p{2cm}|p{1.5cm}|p{1.5cm}|p{1.5cm}|p{1.5cm}||}\hline
\hline
\textbf{algebra}&
\textbf{GCM}&
Finite&
Affine&
Indefinite Hyperbolic&
Indefinite Not Hyperbolic
\\ \hline

$\mathfrak{g}_{6,18}$
&$\left( \begin{smallmatrix}
2&-3\\-1&2
\end{smallmatrix}\right)$
&         $G_2$             
&                      
&                     
&                      
\\

$\mathfrak{g}_{6,19}$
&$\left( \begin{smallmatrix}
2&-4\\-2&2
\end{smallmatrix}\right)$
&                     
&                      
&       (4,2)              
&                      
\\

$\mathfrak{g}_{6,20}$
&$\left( \begin{smallmatrix}
2&-3\\-3&2
\end{smallmatrix}\right)$
&                     
&                      
&      $(3,3)$               
&                     
\\

$\mathfrak{g}_{7,0.1}
{ }^{\ddagger}$
&$\left( \begin{smallmatrix}
2&-5\\-5&2
\end{smallmatrix}\right)$
&                      
&                      
& $(5,5)$                    
&                      
\\

$\mathfrak{g}_{7,0.2}
{ }^{\ddagger}$
&ditto
&                      
&                      
& ditto                    
&                      
\\

$\mathfrak{g}_{7,0.3}
{ }^{\ddagger}$
&ditto
&                      
&                      
& ditto                    
&                      
\\

$\mathfrak{g}_{7,0.4(\lambda)}
{ }^{\ddagger}$
&$\left( \begin{smallmatrix}
2&-4\\-4&2
\end{smallmatrix}\right)$
&                      
&                      
& $(4,4)$                    
&                      
\\
$\mathfrak{g}_{7,0.5}
{ }^{\ddagger}$
&ditto
&                      
&                      
&       ditto              
&                      
\\
$\mathfrak{g}_{7,0.6}
{ }^{\ddagger}$
&$\left( \begin{smallmatrix}
2&-3\\-3&2
\end{smallmatrix}\right)$
&                      
&                      
& $(3,3)$                    
&                      
\\
$\mathfrak{g}_{7,0.7}
{ }^{\ddagger}$
&ditto
&                      
&                      
&   ditto                  
&                      
\\

$\mathfrak{g}_{7,0.8}
{ }^{\ddagger}$
&$\left( \begin{smallmatrix}
2&-3&-3\\-3&2&-3\\-3&-3&2
\end{smallmatrix}\right)$
&                      
&                      
&                     
& $\surd$                     
\\

$\mathfrak{g}_{7,1.01(i)}
{ }^{\ddagger}$
&$\left( \begin{smallmatrix}
2&0&-4\\0&2&-4\\-1&-1&2
\end{smallmatrix}\right)$
&                      
&                      
&  $H_{123}^{(3)}$ (123)                 
&                      
\\

$\mathfrak{g}_{7,1.01(ii)}
{ }^{\ddagger}$
&ditto
&                      
&                      
&  ditto                  
&                      
\\

$\mathfrak{g}_{7,1.02}
{ }^{\ddagger}$
&$\left( \begin{smallmatrix}
2&-2\\-3&2
\end{smallmatrix}\right)
$
&                      
&                      
&$(3,2)$                    
&                      
\\

$\mathfrak{g}_{7,1.03}
{ }^{\ddagger}$
&$\left( \begin{smallmatrix}
2&-3\\-2&2
\end{smallmatrix}\right)
$
&                      
&                      
&  $(3,2)$                  
&                      
\\

$\mathfrak{g}_{7,1.1(i_\lambda)}$
$\lambda \neq 0$
&$\left( \begin{smallmatrix}
2&-5\\-3&2
\end{smallmatrix}\right)
$
&                      
&                      
& $(5,3)$                    
&                      
\\
$\mathfrak{g}_{7,1.1(i_\lambda)}$
$\lambda = 0$
&$\left( \begin{smallmatrix}
2&-5\\-2&2
\end{smallmatrix}\right)
$
&                      
&                      
& $(5,2)$                    
&                      
\\

$\mathfrak{g}_{7,1.1(ii)}$
&$\left( \begin{smallmatrix}
2&-5\\-1&2
\end{smallmatrix}\right)
$
&                      
&                     
&       $(5,1)$               
&                      
\\

$\mathfrak{g}_{7,1.1(iii)}$
&$\left( \begin{smallmatrix}
2&-4\\-3&2
\end{smallmatrix}\right)
$
&                      
&                      
& $(4,3)$                    
&                      
\\

$\mathfrak{g}_{7,1.1(iv)}$
&$\left( \begin{smallmatrix}
2&-2\\-3&2
\end{smallmatrix}\right)
$
&                      
&                      
& $(3,2)$                    
&                      
\\

$\mathfrak{g}_{7,1.1(v)}$
&$\left( \begin{smallmatrix}
2&0&-4\\0&2&-2\\-2&-1&2
\end{smallmatrix}\right)
$
&                      
&                      
&                     
& $\surd$                     
\\

$\mathfrak{g}_{7,1.1(vi)}$
&$\left( \begin{smallmatrix}
2&-3&-1\\-3&2&0\\-1&0&2
\end{smallmatrix}\right)
$
&                      
&                      
&                     
& $\surd$                     
\\

$\mathfrak{g}_{7,1.2(i_\lambda)}
{ }^{\ddagger}$
&$\left( \begin{smallmatrix}
2&-3&-2\\-3&2&-2\\-1&-1&2
\end{smallmatrix}\right)
$
&                      
&                     
&                    
& $\surd$                     
\\

$\mathfrak{g}_{7,1.2(ii)}
{ }^{\ddagger}$
&ditto
&                      
&                     
&                    
& ditto                     
\\

$\mathfrak{g}_{7,1.2(iii)}
{ }^{\ddagger}$
&ditto
&                      
&                     
&                    
& ditto                     
\\

$\mathfrak{g}_{7,1.2(iv)}
{ }^{\ddagger}$
&ditto
&                      
&                     
&                    
&  ditto                  
\\
\hline
\hline

\end{tabular}
}
\end{flushleft}
\end{table}

\setcounter{table}{1}
\begin{table}[h]%
\caption{continued}
\begin{flushleft}
{\fontsize{8}{7.5} \selectfont
\begin{tabular}{||p{1.4cm}|p{2cm}|p{1.5cm}|p{1.5cm}|p{1.5cm}|p{1.5cm}||}\hline
\hline
\textbf{algebra}&
\textbf{GCM}&
Finite&
Affine&
Indefinite Hyperbolic&
Indefinite Not Hyperbolic
\\ \hline

$\mathfrak{g}_{7,1.3(i_\lambda)}
{ }^{\ddagger}$
&$\left( \begin{smallmatrix}
2&-3&-3\\-2&2&-1\\-2&-1&2
\end{smallmatrix}\right)
$
&                      
&                     
&                    
& $\surd$                     
\\

$\mathfrak{g}_{7,1.3(ii)}
{ }^{\ddagger}$
&ditto
&                      
&                     
&                    
& ditto                     
\\

$\mathfrak{g}_{7,1.3(iii)}
{ }^{\ddagger}$
&$\left( \begin{smallmatrix}
2&-3&-3\\-2&2&0\\-2&0&2
\end{smallmatrix}\right)
$
&                      
&                     
&                    
& $\surd$                     
\\

$\mathfrak{g}_{7,1.3(iv)}
{ }^{\ddagger}$
&$\left( \begin{smallmatrix}
2&-2&-2\\-2&2&-1\\-2&-1&2
\end{smallmatrix}\right)
$
&                      
&                     
&  $H_{18}^{(3)}$ \;(40)                 
&                      
\\
$\mathfrak{g}_{7,1.3(v)}
{ }^{\ddagger}$
&$\left( \begin{smallmatrix}
2&-3&-3&-2\\-2&2&-1&-1\\-2&-1&2&-1\\-1&-1&-1&2
\end{smallmatrix}\right)
$
&                      
&                     
&                    
& $\surd$                     
\\

$\mathfrak{g}_{7,1.4}$
&$\left( \begin{smallmatrix}
2&-5\\-2&2
\end{smallmatrix}\right)
$
&                      
&                     
& $(5,2)$                  
&                      
\\

$\mathfrak{g}_{7,1.5}$
&$\left( \begin{smallmatrix}
2&-4\\-2&2
\end{smallmatrix}\right)
$
&                      
&                     
&  $(4,2)$                 
&                      
\\

$\mathfrak{g}_{7,1.6}$
&$\left( \begin{smallmatrix}
2&-5\\-2&2
\end{smallmatrix}\right)
$
&                      
&                     
&  $(5,2)$                 
&                      
\\

$\mathfrak{g}_{7,1.7}$
&$\left( \begin{smallmatrix}
2&-2&-2\\-2&2&-1\\-1&-1&2
\end{smallmatrix}\right)
$
&                      
&                     
&  $H_{8}^{(3)}$   \; (34)               
&                      
\\

$\mathfrak{g}_{7,1.8}$
&$\left( \begin{smallmatrix}
2&-3&0\\-2&2&-1\\0&-2&2
\end{smallmatrix}\right)
$
&                      
&                     
&                   
&  $\surd$                    
\\

$\mathfrak{g}_{7,1.9}$
&$\left( \begin{smallmatrix}
2&-3&-1\\-2&2&-1\\-1&-1&2
\end{smallmatrix}\right)
$
&                      
&                     
&                   
&  $\surd$                    
\\

$\mathfrak{g}_{7,1.10}$
&$\left( \begin{smallmatrix}
2&-4\\-3&2
\end{smallmatrix}\right)
$
&                      
&                     
&$(4,3)$                   
&                      
\\

$\mathfrak{g}_{7,1.11}
{ }^{\ddagger}$
&$\left( \begin{smallmatrix}
2&-4&-3\\-2&2&-1\\-1&-1&2
\end{smallmatrix}\right)
$
&                      
&                     
&                   
&  $\surd$                    
\\

$\mathfrak{g}_{7,1.12}
{ }^{\ddagger}$
&$\left( \begin{smallmatrix}
2&-4&-2\\-2&2&-1\\-1&-1&2
\end{smallmatrix}\right)
$
&                      
&                     
&                   
&  $\surd$                    
\\

$\mathfrak{g}_{7,1.13}$
&$\left( \begin{smallmatrix}
2&-4\\-2&2
\end{smallmatrix}\right)
$
&                      
&                     
& $(4,2)$                  
&                      
\\

$\mathfrak{g}_{7,1.14}$
&$\left( \begin{smallmatrix}
2&-3\\-3&2
\end{smallmatrix}\right)
$
&                      
&                     
&  $(3,3)$                 
&                     
\\

$\mathfrak{g}_{7,1.15}
{ }^{\ddagger}$
&$\left( \begin{smallmatrix}
2&-4&-3\\-2&2&-1\\-1&-1&2
\end{smallmatrix}\right)
$
&                      
&                     
&                   
&  $\surd$                    
\\

$\mathfrak{g}_{7,1.16}
{ }^{\ddagger}$
&$\left( \begin{smallmatrix}
2&-3&-2\\-2&2&-1\\-1&-1&2
\end{smallmatrix}\right)
$
&                      
&                     
&                   
&  $\surd$                    
\\

$\mathfrak{g}_{7,1.17}
{ }^{\ddagger}$
&$\left( \begin{smallmatrix}
2&-4\\-4&2
\end{smallmatrix}\right)
$
&                      
&                     
& $(4,4)$                  
&                      
\\

$\mathfrak{g}_{7,1.18}
{ }^{\ddagger}$
&$\left( \begin{smallmatrix}
2&-3&-2\\-2&2&-1\\-1&-1&2
\end{smallmatrix}\right)
$
&                      
&                     
&                   
&  $\surd$                    
\\

$\mathfrak{g}_{7,1.19}
{ }^{\ddagger}$
&$\left( \begin{smallmatrix}
2&-2&-2\\-2&2&-2\\-2&-2&2
\end{smallmatrix}\right)
$
&                      
&                     
&  $H_{71}^{(3)}$  \; (80)                
&                      
\\

$\mathfrak{g}_{7,1.20}$
&$\left( \begin{smallmatrix}
2&-1&-2\\-3&2&-1\\-1&-1&2
\end{smallmatrix}\right)
$
&                      
&                  
& $H^{(3)}_{9}$ \; (6)                   
&                      
\\

$\mathfrak{g}_{7,1.21}
{ }^{\ddagger}$
&$\left( \begin{smallmatrix}
2&-3&-2\\-3&2&-2\\-1&-1&2
\end{smallmatrix}\right)
$
&                      
&                     
&                    
& $\surd$                     

\\
\hline
\hline

\end{tabular}
}
\end{flushleft}
\end{table}

\setcounter{table}{1}
\begin{table}[h]%
\caption{continued}
\begin{flushleft}
{\fontsize{8}{7.5} \selectfont
\begin{tabular}{||p{1.4cm}|p{2cm}|p{1.5cm}|p{1.5cm}|p{1.5cm}|p{1.5cm}||}\hline
\hline
\textbf{algebra}&
\textbf{GCM}&
Finite&
Affine&
Indefinite Hyperbolic&
Indefinite Not Hyperbolic
\\ \hline

$\mathfrak{g}_{7,2.1(i_\lambda)}$
&$\left( \begin{smallmatrix}
2&-3&-1\\-1&2&-1\\-1&-1&2
\end{smallmatrix}\right)
$
&                      
&                      
& $ H^{(3)}_3$ \; (2)                   
&                      
\\

$\mathfrak{g}_{7,2.1(ii)}$
&$\left( \begin{smallmatrix}
2&-3&-1\\-1&2&0\\-1&0&2
\end{smallmatrix}\right)
$
&                      
& $D_4^{(3)}$                     
&                    
&                      
\\

$\mathfrak{g}_{7,2.1(iii)}$
&$\left( \begin{smallmatrix}
2&-3&0&0\\-1&2&-1&-1\\0&-1&2&0\\0&-1&0&2
\end{smallmatrix}\right)
$
&                      
&                      
& $H^{(4)}_{27}$ \, (150)                   
&                      
\\

$\mathfrak{g}_{7,2.1(iv)}$
&$\left( \begin{smallmatrix}
2&0&-1&-2\\0&2&-1&0\\-1&-1&2&-1\\-1&0&-1&2
\end{smallmatrix}\right)
$
&                      
&                      
&                    
&  $\surd$                    
\\

$\mathfrak{g}_{7,2.1(v)}$
&$\left( \begin{smallmatrix}
2&-2&-1\\-1&2&-1\\-1&-1&2
\end{smallmatrix}\right)
$
&                      
&                      
&  $H_1^{(3)}$ \; (1)                 
&                      
\\

$\mathfrak{g}_{7,2.2}
{ }^{\ddagger}$
&$\left( \begin{smallmatrix}
2&-1&-1\\-2&2&-1\\-2&-1&2
\end{smallmatrix}\right)
$
&                      
&                      
& $ H^{(3)}_7$ \; (4)                   
&                      
\\

$\mathfrak{g}_{7,2.3}$
&$\left( \begin{smallmatrix}
2&-5\\-1&2
\end{smallmatrix}\right)
$
&                      
&                      
&  $(5,1)$                    
&                      
\\

$\mathfrak{g}_{7,2.4}$
&$\left( \begin{smallmatrix}
2&-4\\-1&2
\end{smallmatrix}\right)
$
&                      
&  $A_2^{(2)}$                    
&                      
&                      
\\

$\mathfrak{g}_{7,2.5}$
&$\left( \begin{smallmatrix}
2&-2\\-2&2
\end{smallmatrix}\right)
$
&                      
&  $A_1^{(1)}$                    
&                      
&                      
\\

$\mathfrak{g}_{7,2.6}$
&$\left( \begin{smallmatrix}
2&-3\\-2&2
\end{smallmatrix}\right)
$
&                      
&                      
&  $(3,2)$                    
&                      
\\

$\mathfrak{g}_{7,2.7}$
&$\left( \begin{smallmatrix}
2&-4\\-2&2
\end{smallmatrix}\right)
$
&                      
&                      
&  $(4,2)$                    
&                      
\\

$\mathfrak{g}_{7,2.8}$
&$\left( \begin{smallmatrix}
2&-3\\-2&2
\end{smallmatrix}\right)
$
&                      
&                      
&  $(3,2)$                    
&                      
\\

$\mathfrak{g}_{7,2.9}$
&$\left( \begin{smallmatrix}
2&-3\\-3&2
\end{smallmatrix}\right)
$
&                      
&                      
&  $(3,3)$                    
&                      
\\

$\mathfrak{g}_{7,2.10}$
&$\left( \begin{smallmatrix}
2&-3&-1\\-1&2&0\\-1&0&2
\end{smallmatrix}\right)
$
&                      
&  $D_4^{(3)}$                    
&                      
&                      
\\

$\mathfrak{g}_{7,2.11}
{ }^{\ddagger}$
&$\left( \begin{smallmatrix}
2&-3&-2\\-2&2&-1\\-1&-1&2
\end{smallmatrix}\right)
$
&                      
&                     
&                      
&  $\surd$                    
\\

$\mathfrak{g}_{7,2.12}
{ }^{\ddagger}$
&$\left( \begin{smallmatrix}
2&-2&-2\\-2&2&0\\-2&0&2
\end{smallmatrix}\right)
$
&                      
&                      
&  $H^{(3)}_{109}$   (112)                  
&                      
\\

$\mathfrak{g}_{7,2.13}$
&$\left( \begin{smallmatrix}
2&-3&0\\-1&2&-2\\0&-1&2
\end{smallmatrix}\right)
$
&                      
&                      
&  $H^{(3)}_{100}$  \; (26)                  
&                      
\\

$\mathfrak{g}_{7,2.14}$
&$\left( \begin{smallmatrix}
2&-4&0\\-1&2&-1\\0&-2&2
\end{smallmatrix}\right)
$
&                      
&                      
&  $H^{(3)}_{107}$   (111)                  
&                      
\\

$\mathfrak{g}_{7,2.15}$
&$\left( \begin{smallmatrix}
2&-4&0\\-1&2&-1\\0&-1&2
\end{smallmatrix}\right)
$
&                      
&                      
&  $H^{(3)}_{97}$    \, (104)                
&                      
\\

$\mathfrak{g}_{7,2.16}$
&
\text{ ditto }
&                      
&                      
& \text{ ditto }                    
&                      
\\

$\mathfrak{g}_{7,2.17}$
&$\left( \begin{smallmatrix}
2&-3&0\\-2&2&-1\\0&-1&2
\end{smallmatrix}\right)
$
&                      
&                      
&                      
& $\surd$                     
\\

$\mathfrak{g}_{7,2.18}$
&
\text{ ditto }
&                      
&                      
&                      
&
\text{ ditto }
\\

$\mathfrak{g}_{7,2.19}$
&$\left( \begin{smallmatrix}
2&-3&-1\\-2&2&0\\-1&0&2
\end{smallmatrix}\right)
$
&                      
&                      
&                      
& $\surd$                     
\\

$\mathfrak{g}_{7,2.20}$
&$\left( \begin{smallmatrix}
2&-1&-3\\-1&2&0\\-2&0&2
\end{smallmatrix}\right)
$
&                      
&                      
&                      
& $\surd$                     
\\
$\mathfrak{g}_{7,2.21}$
&$\left( \begin{smallmatrix}
2&-3&-1\\-1&2&-1\\-1&-1&2
\end{smallmatrix}\right)
$
&                      
&                      
&  $H^{(3)}_{3}$ \; (2)                   
&                      
\\

$\mathfrak{g}_{7,2.22}$
&$\left( \begin{smallmatrix}
2&0&-3\\0&2&-1\\-2&-1&2
\end{smallmatrix}\right)
$
&                      
&                      
&                      
& $\surd$                     
\\

$\mathfrak{g}_{7,2.23}$
&$\left( \begin{smallmatrix}
2&0&0&-2\\0&2&-2&0\\
0&-1&2&-1\\-1&0&-1&2
\end{smallmatrix}\right)
$
&                      
&  $D_4^{(2)}$                    
&                      
&                      

\\
\hline
\hline

\end{tabular}
}
\end{flushleft}
\end{table}
\setcounter{table}{1}
\begin{table}[h]%
\begin{flushleft}
\caption{ continued}
{\fontsize{8}{7.5} \selectfont
\begin{tabular}{||p{1.4cm}|p{2cm}|p{1.5cm}|p{1.5cm}|p{1.5cm}|p{1.5cm}||}\hline
\hline
\textbf{algebra}&
\textbf{GCM}&
Finite&
Affine&
Indefinite Hyperbolic&
Indefinite Not Hyperbolic
\\ \hline

$\mathfrak{g}_{7,2.24}$
&$\left( \begin{smallmatrix}
2&-3&0\\-1&2&-1\\0&-1&2
\end{smallmatrix}\right)
$
&                      
&  $G_2^{(1)}$                    
&                      
&                      
\\

$\mathfrak{g}_{7,2.25}
{ }^{\ddagger}$
&$\left( \begin{smallmatrix}
2&-3&0&-2\\-1&2&-1&0\\0&-1&2&-1\\-1&0&-1&2
\end{smallmatrix}\right)
$
&                      
&                      
&                      
& $\surd$                     
\\

$\mathfrak{g}_{7,2.26}
{ }^{\ddagger}$
&$\left( \begin{smallmatrix}
2&-2&-1\\-2&2&-1\\-1&-1&2
\end{smallmatrix}\right)
$
&                      
&                      
&   $H_2^{(3)}$ \; (32)            
&                      
\\

$\mathfrak{g}_{7,2.27}
{ }^{\ddagger}$
&$\left( \begin{smallmatrix}
2&-2&-1&-1\\-2&2&0&-1\\-1&0&2&0\\-1&-1&0&2
\end{smallmatrix}\right)
$
&                      
&                      
&                      
& $\surd$                     
\\
$\mathfrak{g}_{7,2.28}$
&$\left( \begin{smallmatrix}
2&-1&-2&0\\-2&2&0&0\\-1&0&2&-1\\0&0&-1&2
\end{smallmatrix}\right)
$
&                      
&                      
&  $H^{(4)}_{40}$  \, (164)                  
&                      
\\

$\mathfrak{g}_{7,2.29}$
&$\left( \begin{smallmatrix}
2&-2&0&0\\-2&2&-1&0\\0&-1&2&-1\\0&0&-1&2
\end{smallmatrix}\right)
$
&                      
&                      
&                      
& $\surd$                     
\\

$\mathfrak{g}_{7,2.30}$
&$\left( \begin{smallmatrix}
2&-3&0&0\\-2&2&0&0\\0&0&2&-1\\0&0&-1&2
\end{smallmatrix}\right)
$
&                      
&                      
& $(3,2) \times A_2$                     
&                     
\\
$\mathfrak{g}_{7,2.31}$
&$\left( \begin{smallmatrix}
2&-3&0\\-1&2&-2\\0&-1&2
\end{smallmatrix}\right)
$
&                      
&                      
&  $H^{(3)}_{100}$  \; (26)                  
&                      
\\

$\mathfrak{g}_{7,2.32}$
&$\left( \begin{smallmatrix}
2&-3&-1\\-1&2&0\\-2&0&2
\end{smallmatrix}\right)
$
&                      
&                      
&  $H^{(3)}_{106}$   \; (25)                 
&                      
\\

$\mathfrak{g}_{7,2.33}$
&$\left( \begin{smallmatrix}
2&-3&0\\-1&2&-1\\0&-2&2
\end{smallmatrix}\right)
$
&                      
&                      
&  $H^{(3)}_{105}$  \; (28)                  
&                      
\\

$\mathfrak{g}_{7,2.34}$
&$\left( \begin{smallmatrix}
2&-2&-1\\-2&2&0\\-2&0&2
\end{smallmatrix}\right)
$
&                      
&                      
&  $H^{(3)}_{104}$      (107)               
&                      
\\

$\mathfrak{g}_{7,2.35}$
&$\left( \begin{smallmatrix}
2&-1&-2\\-2&2&0\\-1&0&2
\end{smallmatrix}\right)
$
&                      
& $A_4^{(2)}$                     
&                     
&                      
\\

$\mathfrak{g}_{7,2.36}$
&$\left( \begin{smallmatrix}
2&0&-1&-1\\0&2&-1&-1\\-1&-2&2&0\\-1&-1&0&2
\end{smallmatrix}\right)
$
&                      
&                      
&  $H^{(4)}_8$ \, (131)                    
&                     
\\

$\mathfrak{g}_{7,2.37}
{ }^{\ddagger}$
&$\left( \begin{smallmatrix}
2&-2&-1\\-2&2&-1\\-1&-1&2
\end{smallmatrix}\right)
$
&                      
&                      
&  $H^{(3)}_{2}$  \, (32)                  
&                      
\\

$\mathfrak{g}_{7,2.38}$
&$\left( \begin{smallmatrix}
2&-2&0&-1\\-2&2&-1&0\\0&-1&2&0\\-1&0&0&2
\end{smallmatrix}\right)
$
&                      
&                      
&                      
&  $\surd$                    
\\

$\mathfrak{g}_{7,2.39}$
&$\left( \begin{smallmatrix}
2&-2&-2\\-1&2&-1\\-1&-1&2
\end{smallmatrix}\right)
$
&                      
&                      
&  $H^{(3)}_{5}$ \; (3)                   
&                      
\\

$\mathfrak{g}_{7,2.40}$
&$\left( \begin{smallmatrix}
2&-2&-1\\-2&2&-1\\-1&-1&2
\end{smallmatrix}\right)
$
&                      
&                      
&  $H^{(3)}_{2}$  \; (32)                  
&                      
\\

$\mathfrak{g}_{7,2.41}$
&$\left( \begin{smallmatrix}
2&-2&-2\\-2&2&0\\-1&0&2
\end{smallmatrix}\right)
$
&                      
&                      
&  $H^{(3)}_{99}$     \, (106)               
&                      
\\

$\mathfrak{g}_{7,2.42}$
&$\left( \begin{smallmatrix}
2&-1&-2\\-2&2&-1\\-1&-1&2
\end{smallmatrix}\right)
$
&                      
&                      
&  $H^{(3)}_{6}$     \; (5)               
&                      
\\

$\mathfrak{g}_{7,2.43}$
&$\left( \begin{smallmatrix}
2&-2&-2\\-1&2&0\\-2&0&2
\end{smallmatrix}\right)
$
&                      
&                      
&  $H^{(3)}_{99}$  \, (106)                  
&                      
\\

$\mathfrak{g}_{7,2.44}$
&$\left( \begin{smallmatrix}
2&-1&-2\\-2&2&-1\\-1&-1&2
\end{smallmatrix}\right)
$
&                      
&                      
&  $H^{(3)}_{6}$  \; (5)                  
&                      
\\

$\mathfrak{g}_{7,2.45}$
&$\left( \begin{smallmatrix}
2&-2&0&-1\\-1&2&-1&-1\\0&-1&2&0\\-1&-1&0&2
\end{smallmatrix}\right)
$
&                      
&                      
&                     
&  $\surd$                    
\\
 \hline
 \hline
\end{tabular}
}
\end{flushleft}
\end{table}
\setcounter{table}{1}
\begin{table}[h]%
\begin{flushleft}
\caption{continued}
{\fontsize{8}{7.5} \selectfont
\begin{tabular}{||p{1.4cm}|p{2.5cm}|p{1.5cm}|p{1.5cm}|p{1.5cm}|p{1.5cm}||}\hline
\hline
\textbf{algebra}&
\textbf{GCM}&
Finite&
Affine&
Indefinite Hyperbolic&
Indefinite Not Hyperbolic
\\ \hline

$\mathfrak{g}_{7,3.1(i_\lambda)}$
&$\left( \begin{smallmatrix}
2&-1&-1\\-1&2&-1\\-1&-1&2
\end{smallmatrix}\right) $
&                      
&  $A_2^{(1)}$                    
&                     
&                     
\\

$\mathfrak{g}_{7,3.1(iii)}$
&$\left( \begin{smallmatrix}
2&-1&-1&-1\\-1&2&0&0\\-1&0&2&0\\-1&0&0&2
\end{smallmatrix}\right) $
&  $D_4$                    
&                      
&                     
&                     
\\

$\mathfrak{g}_{7,3.2}$
&$\left( \begin{smallmatrix}
2&-3&-1\\-1&2&0\\-1&0&2
\end{smallmatrix}\right) $
&                      
&  $D_4^{(3)}$                    
&                     
&                     
\\

$\mathfrak{g}_{7,3.3}$
&$\left( \begin{smallmatrix}
2&-3&0\\-1&2&-1\\0&-1&2
\end{smallmatrix}\right) $
&                      
&  $G_2^{(1)}$                    
&                     
&                     
\\

$\mathfrak{g}_{7,3.4}$
&$\left( \begin{smallmatrix}
2&-1&-1\\-2&2&0\\-2&0&2
\end{smallmatrix}\right) $
&                      
&  $D_3^{(2)}$                    
&                     
&                     
\\

$\mathfrak{g}_{7,3.5}$
&$\left( \begin{smallmatrix}
2&-1&-1\\-2&2&0\\-1&0&2
\end{smallmatrix}\right) $
&  $B_3$                    
&                      
&                     
&                     
\\

$\mathfrak{g}_{7,3.6}$
&$\left( \begin{smallmatrix}
2&-1&-2\\-1&2&-1\\-1&-1&2
\end{smallmatrix}\right) $
&                      
&                      
&  $H^{(3)}_1$ \; (1)            
&                     
\\

$\mathfrak{g}_{7,3.7}$
&$\left( \begin{smallmatrix}
2&-2&0&0\\-1&2&0&-1\\0&0&2&-1\\0&-1&-1&2
\end{smallmatrix}\right) $
&  $B_4$                    
&                      
&                     
&                     
\\

$\mathfrak{g}_{7,3.8}$
&$\left( \begin{smallmatrix}
2&-2&-1&0\\-1&2&0&-1\\-1&0&2&0\\0&-1&0&2
\end{smallmatrix}\right) $
&  $F_4$                    
&                      
&                     
&                     
\\

$\mathfrak{g}_{7,3.9}$
&$\left( \begin{smallmatrix}
2&-2&0&0\\-1&2&-1&-1\\0&-1&2&0\\0&-1&0&2
\end{smallmatrix}\right) $
&                      
&  $B_3^{(1)}$                    
&                     
&                     
\\

$\mathfrak{g}_{7,3.10}$
&$\left( \begin{smallmatrix}
2&-1&-1&0\\-1&2&0&-2\\-1&0&2&0\\0&-1&0&2
\end{smallmatrix}\right) $
&  $C_4$                    
&                      
&                     
&                     
\\

$\mathfrak{g}_{7,3.11}$
&$\left( \begin{smallmatrix}
2&-2&-1&0\\-1&2&0&-1\\-1&0&2&0\\0&-1&0&2
\end{smallmatrix}\right) $
&  $F_4$                    
&                      
&                     
&                     
\\

$\mathfrak{g}_{7,3.12}$
&$\left( \begin{smallmatrix}
2&-1&-1&0\\-1&2&0&-1\\-1&0&2&-1\\0&-1&-1&2
\end{smallmatrix}\right) $
&                      
&  $A^{(1)}_3$                    
&                     
&                     
\\

$\mathfrak{g}_{7,3.13}$
&$\left( \begin{smallmatrix}
2&-2&0&0\\-2&2&0&0\\0&0&2&-1\\0&0&-1&2
\end{smallmatrix}\right) $
&                      
&  $A^{(1)}_1 \times A_2$                    
&                     
&                     
\\

$\mathfrak{g}_{7,3.14}$
&$\left( \begin{smallmatrix}
2&-1&-2&0\\-1&2&0&-1\\-1&0&2&0\\0&-1&0&2
\end{smallmatrix}\right) $
&  $C_4$                    
&                      
&                     
&                     
\\

$\mathfrak{g}_{7,3.15}$
&$\left( \begin{smallmatrix}
2&-1&-1&0\\-2&2&0&0\\-1&0&2&-1\\0&0&-1&2
\end{smallmatrix}\right) $
&  $B_4$                    
&                      
&                     
&                     
\\

$\mathfrak{g}_{7,3.16}$
&$\left( \begin{smallmatrix}
2&-2&0&0\\-1&2&0&0\\0&0&2&-2\\0&0&-1&2
\end{smallmatrix}\right) $
&  $B_2 \times B_2$                    
&                      
&                     
&                     
\\

$\mathfrak{g}_{7,3.17}$
&$\left( \begin{smallmatrix}
2&-3&0&0\\-1&2&0&0\\0&0&2&-1\\0&0&-1&2
\end{smallmatrix}\right) $
&  $G_2 \times A_2$                    
&                      
&                     
&                     
\\

$\mathfrak{g}_{7,3.18}$
&$\left( \begin{smallmatrix}
2&-2&0&0&0\\-1&2&0&0&-1\\0&0&2&-1&0\\0&0&-1&2&0\\0&-1&0&0&2
\end{smallmatrix}\right) $
&  $B_3 \times A_2$                    
&                      
&                     
&                     
\\

$\mathfrak{g}_{7,3.19}$
&$\left( \begin{smallmatrix}
2&-1&-1&0&0\\-1&2&0&0&0\\-1&0&2&-1&0\\0&0&-1&2&-1\\0&0&0&-1&2
\end{smallmatrix}\right) $
&  $A_5$                    
&                      
&                     
&                     
\\
 \hline
 \hline
\end{tabular}
}
\end{flushleft}
\end{table}
\setcounter{table}{1}
\begin{table}[h]%
\begin{flushleft}
\caption{continued}
{\fontsize{8}{7.5} \selectfont
\begin{tabular}{||p{1.4cm}|p{2.8cm}|p{1.8cm}|p{1.5cm}|p{1.5cm}|p{1.5cm}||}\hline
\hline
\textbf{algebra}&
\textbf{GCM}&
Finite&
Affine&
Indefinite Hyperbolic&
Indefinite Not Hyperbolic
\\ \hline

$\mathfrak{g}_{7,3.20}$
&$\left( \begin{smallmatrix}
2&-2&-2\\-1&2&0\\-1&0&2
\end{smallmatrix}\right) $
&                      
&  $C_2^{(1)}$                    
&                     
&                     
\\

$\mathfrak{g}_{7,3.21}$
&$\left( \begin{smallmatrix}
2&-2&-1\\-1&2&0\\-2&0&2
\end{smallmatrix}\right) $
&                      
&  $A_4^{(2)}$                    
&                     
&                     
\\

$\mathfrak{g}_{7,3.22}$
&$\left( \begin{smallmatrix}
2&-1&-2\\-1&2&0\\-1&0&2
\end{smallmatrix}\right) $
&  $C_3$                    
&                      
&                     
&                     
\\

$\mathfrak{g}_{7,3.23}$
&$\left( \begin{smallmatrix}
2&-2&-1\\-2&2&0\\-1&0&2
\end{smallmatrix}\right) $
&                      
&                      
&   $H^{(3)}_{96}$ \, (103)                 
&                     
\\

$\mathfrak{g}_{7,3.24}$
&$\left( \begin{smallmatrix}
2&-1&0&0\\-1&2&-1&-1\\0&-1&2&-1\\0&-1&-1&2
\end{smallmatrix}\right) $
&                      
&                      
&   $H^{(4)}_{3}$ \, (126)                 
&                     
\\

$\mathfrak{g}_{7,4.1}$
&$\left( \begin{smallmatrix}
2&-1&-1&0\\-1&2&0&0\\-1&0&2&-1\\0&0&-1&2
\end{smallmatrix}\right) $
&   $A_4$                   
&                      
&                     
&                     
\\

$\mathfrak{g}_{7,4.2}$
&$\left( \begin{smallmatrix}
2&-1&-1&-1\\-1&2&0&0\\-1&0&2&0\\-1&0&0&2
\end{smallmatrix}\right) $
&   $D_4$                   
&                      
&                     
&                     
\\

$\mathfrak{g}_{7,4.3}$
&$\left( \begin{smallmatrix}
2&-1&0&0&0\\-1&2&0&0&0\\0&0&2&0&-1\\0&0&0&2&-1\\0&0&-1&-1&2
\end{smallmatrix}\right) $
&   $A_2 \times A_3$                   
&                      
&                     
&                     
\\

$\mathfrak{g}_{7,4.4}$
&$\left( \begin{smallmatrix}
2&0&0&-1&0&0\\
0&2&0&0&-1&0\\
0&0&2&0&0&-1\\
-1&0&0&2&0&0\\
0&-1&0&0&2&0\\
0&0&-1&0&0&2
\end{smallmatrix}\right) $
&   $A_2 \times A_2 \times A_2$                   
&                      
&                     
&                     
\\
 \hline
 \hline
\end{tabular}
}
\end{flushleft}
\end{table}

\end{document}